\documentclass[11pt,reqno,a4paper]{amsart}

\usepackage[utf8]{inputenc}
\usepackage{graphicx}
\usepackage{color}
\usepackage{transparent}

\usepackage{dsfont}
\usepackage{gensymb}
\usepackage[active]{srcltx}
\usepackage[colorlinks,pdfpagelabels,pdfstartview = FitH,bookmarksopen = true,bookmarksnumbered = true,linkcolor = blue,plainpages = false,hypertexnames = false,citecolor = red] {hyperref}
\usepackage{mathrsfs}
\usepackage{amssymb,amsmath, amsfonts, amsthm}
\usepackage{latexsym}
\usepackage[dvips]{epsfig}
\newtheorem{theorem}{Theorem}[section]
\newtheorem{proposition}[theorem]{Proposition}
\newtheorem{lemma}[theorem]{Lemma}
\newtheorem{corollary}[theorem]{Corollary}
\newtheorem{remark}[theorem]{Remark}
\newtheorem{definition}[theorem]{Definition}



\usepackage{epsfig,color,xcolor}

\newcommand{\ble}{\begin{lemma}}
\newcommand{\ele}{\end{lemma}}
\newcommand{\be}{\begin{equation*}}
\newcommand{\ee}{\end{equation*}}

\newcommand{\bel}{\begin{equation}}
\newcommand{\eel}{\end{equation}}

\newcommand{\ep}{\varepsilon}
\newcommand{\fr}{\frac }

\newcommand{\lap}{\Delta}
\newcommand{\N}{\mathbb{N}}
\newcommand{\na}{\nabla}

\newcommand{\R}{\mathbb{R}}

\renewcommand{\to}{\rightarrow}
\newcommand{\To}{\longrightarrow}

\newcommand{\xip}{x_{i,p}}
\newcommand{\xjp}{x_{j,p}}

\newcommand{\mip}{\mu_{i,p}}

\newcommand{\upp}{u_p}

\def\sideremark#1{\ifvmode\leavevmode\fi\vadjust{\vbox to0pt{\vss% the remark
 \hbox to 0pt{\hskip\hsize\hskip1em%                          will appear only
 \vbox{\hsize2.1cm\tiny\raggedright\pretolerance10000%          on the side
  \noindent #1\hfill}\hss}\vbox to15pt{\vfil}\vss}}}%

\newcommand{\virg}[1]{\textquotedblleft#1\textquotedblright}
\numberwithin{equation}{section}
\begin{document}

\title[]{Asymptotic analysis for the Lane-Emden\\ problem in dimension two}

\author[]{Francesca De Marchis, Isabella Ianni, Filomena Pacella}

\address{Francesca De Marchis, University of Roma {\em Sapienza}, P.le Aldo Moro 5, 00185 Roma, Italy}
\address{Isabella Ianni, Second University of Napoli, V.le Lincoln 5, 81100 Caserta, Italy}
\address{Filomena Pacella, University of Roma {\em Sapienza}, P.le Aldo Moro 5, 00185 Roma, Italy}

\thanks{2010 \textit{Mathematics Subject classification:} 35B05, 35B06, 35J91. }

\thanks{ \textit{Keywords}: semilinear elliptic equations, superlinear elliptic boundary value problems, asymptotic analysis, concentration of solutions.}

\thanks{Research supported by: PRIN $201274$FYK7$\_005$ grant, INDAM - GNAMPA and Sapienza Research Funds: \virg{Avvio alla ricerca 2015} and \virg{Awards Project 2014}}

\maketitle

\section*{Introduction}\label{section:intro}
We consider the Lane-Emden Dirichlet problem
\begin{equation}\label{problem}\left\{\begin{array}{lr}-\Delta u= |u|^{p-1}u\qquad  \mbox{ in }\Omega\\
u=0\qquad\qquad\qquad\mbox{ on }\partial \Omega
\end{array}\right.
\end{equation}
when $p>1$ and $\Omega\subset\R^2$ is a smooth bounded domain.
The aim of the paper is to survey  some recent results on the asymptotic behavior of solutions of \eqref{problem} as the exponent  $p\rightarrow \infty $. 
\\
We will start in Section \ref{SectionClassicalResults} with a summary of some basic and well known facts about the solutions of \eqref{problem}. We will also describe a recent result about the existence, for $p$ large, of a special class of sign-changing solutions of \eqref{problem} in symmetric domains 
 (see \cite{DeMarchisIanniPacellaJDE}) and we will provide, for $p$ large, the exact computation of the Morse index of least energy nodal radial solutions of \eqref{problem} in the ball, as obtained in \cite{DeMarchisIanniPacellaMORSE}.
\\ 
The asymptotic behavior as $p\rightarrow \infty $, will be described in the Sections \ref{SectionGeneralAnalysis}--\ref{Section:GSymmetric}. In Section \ref{SectionGeneralAnalysis} a general ``profile decomposition'' theorem obtained in \cite{DeMarchisIanniPacellaJEMS} and holding both for positive and sign-changing solutions will be presented with a detailed proof together with some additional new results, recently obtained in \cite{DIPpositive}. Finally in Section \ref{Section:GSymmetric} we will describe the precise limit profile of the symmetric nodal solutions found in
\cite{DeMarchisIanniPacellaJDE} and then studied in \cite{DeMarchisIanniPacellaJEMS}. In particular, the result of this section will show that, asymptotically, as $p\rightarrow \infty $, the solutions look like a superposition of two bubbles with different sign corresponding to radial solutions of the regular and singular Liouville problem in $\mathbb{R}^2$.

\subsection*{Acknowledgments} 

This paper originates from a short course given by F. Pacella at a Conference-School held in Hammamet in March 2015 in honor of Abbas Bahri.
She would like to thank all the organizers for the wonderful and warm hospitality.

\

\

\section{Various results for solutions of the Lane-Emden problem} \label{SectionClassicalResults}

\

We consider the Lane-Emden Dirichlet problem
\begin{equation}\label{problem1}
\left\{\begin{array}{lr}-\Delta u= |u|^{p-1}u\qquad  \mbox{ in }\Omega\\
u=0\qquad\qquad\qquad\mbox{ on }\partial \Omega
\end{array}\right.
\end{equation}
where $p>1$ and $\Omega\subset\R^2$ is a smooth bounded domain.
\\
Since in $2-$dimension any exponent $p>1$ is subcritical (with respect to the Sobolev embedding) it is well known, by standard variational methods, that \eqref{problem1} has at least one positive solution. Moreover, exploiting the oddness of the nonlinearity  $f(u) = |u|^{p-1}u $  and using topological tools it can be proved that \eqref{problem1} admits infinitely many solutions.
\\
It was first proved in \cite{CastroCossioNeuberger}, and later in \cite{BartschWeth} for more general nonlinearities, that there exists at least one solution which changes sign, so it makes sense to study the properties of both positive and sign-changing solutions. The last ones will be often referred as nodal solutions. Among these solutions one can select those which have the least energy, therefore named ``least energy'' (or ``least energy nodal'') solutions. More precisely, considering the energy functional:
\[
E_p(u)=\frac12\int_{\Omega}|\nabla u|^2\,dx-\frac1{p+1}\int_{\Omega}|u|^{p+1}\,dx,\qquad u\in H^1_0(\Omega)
\]
and the Nehari manifold
\[
\mathcal N=\{u\in H^1_0(\Omega)\,:\,\langle E'_p(u),u\rangle=0\}
\]
or the nodal Nehari set
\[
\mathcal N^\pm=\{u\in H^1_0(\Omega)\,:\,\langle E'_p(u),u^\pm\rangle=0\},
\]
where $u^\pm$ are the positive and negative part of $u$, it is possible to prove that the   $ \inf_{\mathcal N} E_p $ (resp. $\inf_{\mathcal N^\pm} E_p$) is achieved. The corresponding minimizers are the least energy positive (resp. nodal) solutions (see \cite{Willem_Book}, \cite{BartschWeth}). Note that any minimizer on $\mathcal N$ cannot change sign and we will assume that it is positive (rather than negative).
\\
Let us observe that $\mathcal N$ is a codimension one manifold in $H^1_0(\Omega)$ while $\mathcal N^\pm$ is a $C^1$-manifold of codimension $2$ in $H^1_0(\Omega)\cap H^2(\Omega)$ (but not in $H^1_0(\Omega)$, see \cite{BartschWeth}).
\\
For the least energy solutions several qualitative properties can be obtained.
\\
We start by considering the case of positive solutions.
\\
Let us first define the Morse index of a solution of \eqref{problem1}.
\begin{definition}\label{def:Morseindex}
The Morse index $m(u)$ of a solution $u$ of \eqref{problem1} is the maximal dimension of a subspace of $C^1_0(\Omega)$ on which the quadratic form
\[
Q(\varphi)=\int_{\Omega}|\nabla\varphi|^2\,dx-p\int_{\Omega}|u|^{p-1}\varphi^2\,dx
\]
is negative definite.
\end{definition}
In the case when $\Omega$ is a bounded domain, $m(u)$ can be equivalently defined as the number of the negative Dirichlet eigenvalue of the linearized operator at $u$:
\[
L_u=-\Delta-p|u|^{p-1}
\]
in the domain $\Omega$.
\\
It is easy to see, just multiplying the equation by $u$ and integrating, that $Q(u) < 0$, so that there is at least one negative direction for $Q(u)$, i.e. $m(u) \geq 1$. This holds for any solution of \eqref{problem1}, either positive or sign-changing. For the least energy solution $u$, since it minimizes the energy on a codimension one manifold, one could guess that $m(u) = 1$. This is what was indeed proved in  \cite{Solimini} (see also \cite{Willem_Book}), for more general nonlinearities.
Another important property of a solution, both for theoretical reasons and for applications, is its symmetry in symmetric domains.
\\
For positive solutions $u$ of \eqref{problem1}, as a consequence of the famous result by Gidas, Ni and Nirenberg \cite{GNN} it holds that if $\Omega$ is symmetric and convex with respect to a line, then $u$ is invariant by reflection with respect to that line. In particular a positive solution of \eqref{problem1} in a ball is radial and strictly radially decreasing.
\\
This result allows to prove that if $\Omega$ is a ball there exists only one positive solution of \eqref{problem1} (this holds also in higher dimension, when $p$ is a subcritical exponent) ( \cite{GNN}, \cite{Srikanth}, \cite{AdiYadava}, \cite{DGP}).
\\
The question of the uniqueness of the positive solution in more general bounded domains is a very difficult one, still open. It has been conjectured (\cite{GNN}) that it should hold in convex domains (also in higher dimension) but, so far, it has only been proved in the case of planar domains symmetric and convex with respect to two orthogonal lines passing through the origin (\cite{DGP}, \cite{PacellaMilan}). If one restricts the question to the least energy solutions (or more generally to solutions of Morse index one) then the uniqueness, in convex planar domains, has been proved in \cite{Lin}. On the other side it is easy to see that there are nonconvex domains for which multiple positive solutions exist; examples of such domains are annular domains or dumbbell domains (\cite{Dancer}, \cite{PacellaMilan}).
\\
More properties of positive solutions and, actually, a good description of their profile, can be obtained, for large exponents $p$, by the asymptotic analysis of the solutions of \eqref{problem1}, as $p\rightarrow \infty$.
\\
This study started in \cite{RenWeiTAMS1994} and \cite{RenWeiPAMS1996} where the authors considered families $(u_p)$ of least energy (hence positive) solutions and, for some domains, proved concentration at a single point, as well as asymptotic estimates, as $p\rightarrow \infty$. Later, inspired by the paper \cite{AdiStruwe}, Adimurthi and Grossi  in \cite{AdiGrossi} identified a ``limit problem'' by showing that suitable rescalings of $u_p$  converge, in $C^2_{loc}(\mathbb{R}^2)$ to a regular solution $U$ of the Liouville problem
\begin{equation}\label{LiouvilleEquationINTRO}
\left\{
\begin{array}{lr}
-\Delta U=e^U\quad\mbox{ in }\R^2\\
\int_{\R^2}e^Udx= 8\pi.
\end{array}
\right.
\end{equation}
They also considered general bounded domains and showed that $\|u_p\|_\infty$ converges to $\sqrt{e}$ as $p\rightarrow \infty$, as it had been previously conjectured.
\\
So the asymptotic profile of the least energy solutions is clear, as well as their energy.
\\
Concerning general positive solutions, a first asymptotic analysis (actually holding for general families of solutions,  both positive and sign-changing)   under the following energy condition:
\begin{equation}\label{energylimitINTRO}
p\int_{\Omega}|\nabla u_p|^2\,dx\leq C
\end{equation}
for some positive constant $C \geq 8 \pi e$ and independent of $p$, was carried out  in \cite{DeMarchisIanniPacellaJEMS}. Then recently in \cite{DIPpositive}, starting  from this,  a complete description of the asymptotic profile of $u_p$ has been obtained (see Section \ref{SectionGeneralAnalysis})  showing that $(u_p)$ concentrates at a finite number of distinct points in $\Omega$, having the limit profile of the solution  $U$ of \eqref{LiouvilleEquationINTRO} when  a suitable rescaling around each of the concentration points is made. Positive solutions with this profile have been found in \cite{EspositoMussoPistoiaPos}.

\

Now let us analyze the case of sign-changing solutions of \eqref{problem1}. 
\\
Since any such solution $u$ has at least two nodal regions (i.e. connected components of the set where $u$ does not vanish), multiplying the equation by $u$ and integrating on each nodal domain, we get that the Morse index $m(u)$ is at least two. For the least energy nodal solution, since it minimizes the energy functional $E_p$ on $\mathcal N^\pm$, it is proved in \cite{BartschWeth} that its Morse index is exactly two.
\\ Concerning symmetry properties of sign-changing solutions, a general result as the one of Gidas, Ni and Nirenberg for positive solutions cannot hold. This is easily understood just thinking of the Dirichlet eigenfunctions of the Laplacian in a ball.
\\
Nevertheless, by using maximum principles, properties of the linearized operator and bounds on the Morse index, partial symmetry results can be obtained also for nodal solutions. This direction of research started in \cite{PacellaJFA} and continued in \cite{PacellaWeth} and \cite{GladialiPacellaWeth}. In particular in these papers, semilinear elliptic equations with nonlinear terms $f(u)$ either convex or with  a convex derivative were studied in rotationally symmetric domains, showing the foliated Schwarz symmetry of solutions (of any sign) having Morse index $m(u) \leq N$, where $N$ is the dimension of the domain. We recall the definition of foliated Schwarz symmetry:
\begin{definition}\label{def:Schwarzsymmetry}
Let $B\subseteq\mathbb{R^N} , N\geq 2$, be a ball or an annulus. A continuous function $v$ in $B$ is said to be foliated Schwarz Symmetric if there exists a unit vector $p \in \mathbb{R}^N$ such that $v(x)$ only depends on $|x|$ and $\vartheta = \arccos (\frac{x}{|x|} \cdot p)$ and is nonincreasing in $\vartheta$.
\end{definition}
In other words a foliated Schwarz symmetric function is axially symmetric and monotone with respect to the angular coordinate.
\\
In particular, in dimension two, the results of \cite{PacellaWeth} allow to claim that, in a ball or in an annulus, any solution $u$ of \eqref{problem1} with Morse index $m(u) \leq 2$ is foliated Schwarz symmetric. Thus, in such domains, the least energy nodal solutions are foliated Schwarz symmetric. 
\\ 
Since radial functions are, obviously, foliated Schwarz symmetric, one may ask whether the least energy nodal solutions are radial or not. The answer to this question was provided by \cite{AftalionPacella} where it was proved that any sign-changing solution $u$ of a semilinear elliptic equation with a general autonomous nonlinearity $f(u)$ in a ball or an annulus must have Morse index $m(u) \geq N + 2$ (again $N$ denotes the dimension of the domain). 
\\
An immediate corollary of this theorem is that, since a least energy nodal solution of \eqref{problem1} has Morse index two, it cannot be radial.
\\
Another interesting consequence of the result of \cite{AftalionPacella}  is that the nodal set of  a least energy nodal solution of \eqref{problem1} in a ball or an annulus must intersect the boundary of $\Omega$. We recall that the nodal set   $ N(u)$ of a function $u$ defined in the domain $\Omega$ is:
\[
N(u)=\overline{\{x\in\Omega\,:\,u(x)=0\}}.
\]
To understand the property of the nodal line is important while studying sign-changing functions. It is an old question related to the study of the nodal eigenfunctions of the Laplacian, in particular of the second eigenfunction. In \cite{Melas} it has been proved that in convex planar domains the nodal set of a second eigenfunction touches the boundary, but the question is still open in higher dimension, except for the case of some symmetric domains ( \cite{LinCMP}, \cite{Damascelli}).
\\
Coming back to nodal solutions of \eqref{problem1} we observe that if $\Omega$ is a ball or an annulus, it is easy to see that there exist both nodal solutions with an interior nodal line and solutions whose nodal line intersects the boundary. Examples of solutions of the first type are the radial ones while of the second type are those which are antisymmetric with respect to a line passing through the center. It is natural to ask whether both kind of solutions exist in more general domains. While it is not difficult to provide examples of symmetric domains where there are nodal solutions whose nodal line intersects the boundary (rectangles, regular polygons etc.) it is not obvious at all that solutions with an interior nodal line exist. In the paper \cite{DeMarchisIanniPacellaJDE} we have succeeded in proving the existence of this type of solutions in some symmetric planar domains, for large exponents $p$. The precise statement is the following:
\begin{theorem}\label{thm:existenceINTRO}
Assume that $\Omega$ is simply connected,  invariant under the action of a finite group $G$ of orthogonal transformations of $\mathbb{R}^2$. If $ |G |\geq 4$ ($|G|$ is the order of the group)  then, for $p$ sufficiently large \eqref{problem1} admits a sign-changing $G$-symmetric solution $u_p$, with two nodal domains, whose nodal line neither touches $\partial \Omega$, nor passes through the origin. Moreover
\[
p\int_{\Omega}|\nabla\upp|^2 dx\leq\alpha\,8\pi e\quad \mbox{for some  $\alpha<5$ and $p$ large.}
\]
\end{theorem}

\

Let us now come back to the question of the Morse index of nodal solutions of \eqref{problem1}. As recalled before, the result of \cite{AftalionPacella} allows to give an estimate from below in the radial case:
\[
m(u) \geq 4
\]
for any radial sign-changing solution $u$ of \eqref{problem1} in a ball or an annulus. In the recent paper \cite{DeMarchisIanniPacellaMORSE} we have been able to compute exactly the Morse index for these solutions when the exponent $p$ is large and $u$ has the least energy among the radial nodal solutions. The result is the following:

\begin{theorem}\label{thm:Morse}
Let $u_p$ be the least energy sign-changing radial solution of \eqref{problem1}. Then
\[
m(u_p) \ =   12
\]
for $p$ sufficiently large.
\end{theorem}
The proof of this theorem is based on a decomposition of the spectrum of the linearized operator at $u_p$, as well as on fine estimates of the radial solution $u_p$ obtained in \cite{GrossiGrumiauPacella2}.
\\
As in the case of positive solutions, a better description of nodal solutions and of their profile can be obtained, for large exponents $p$, by
performing an asymptotic analysis, as $p \rightarrow \infty$. This study started in \cite{GrossiGrumiauPacella1},
by considering a family $(u_p)$ of solutions of \eqref{problem1} satisfying the condition
\[
p\int_{\Omega}|\nabla u|^2\,dx\to16\pi e\qquad\mbox{as $p\to+\infty$}
\]
where $16 \pi e$ is the ``least-asymptotic'' energy for nodal solutions. Under some additional conditions it was proved in \cite{GrossiGrumiauPacella1} that these low-energy solutions concentrate at two distinct points of $\Omega$ and suitable scalings of $u^+_p$ and $u^-_p$ converge to  a regular solution $U$ of \eqref{LiouvilleEquationINTRO}.
\\
Next the case of least-energy radial nodal solutions was considered in \cite{GrossiGrumiauPacella2} where the new phenomenon of
$u^+_p$ and $u^-_p$ concentrating at the same point but with different profile was shown. The precise result is the following:

\begin{theorem}\label{thm:GGP2}

Let $(u_p)$ be a family of least energy radial nodal solutions of \eqref{problem1} in the ball with $u_p$ positive at the center. Then

\begin{itemize}
\item [(i)] a suitable scaling of $u^+_p$ converges in $C^2_{loc} (\mathbb{R}^2)$ to a regular solution $U$ of \eqref{LiouvilleEquationINTRO}

\item [(ii)] a suitable scaling and translation of $u^-_p$ converges in $C^2_{loc} (\mathbb{R}^2 \setminus \{O\})$ to a singular radial solution $V$ of
\begin{equation}\label{SingularLiouvilleEquationINTRO}
\left\{
\begin{array}{lr}
-\Delta V=e^V+H\delta_0\quad\mbox{ in }\R^2\\
\int_{\R^2}e^Vdx<+\infty
\end{array}
\right.
\end{equation}
\end {itemize}
where $H$ is a suitable negative constant and $\delta_0$ is the Dirac measure centered at $O$.
\\
Moreover:
\[
p\int_{\Omega}|\nabla u_p|^2\,dx\to C>16\pi e\qquad\mbox{as $p\to+\infty$}
\]
\end{theorem}
So the theorem shows the existence of solutions which asymptotically look like  a tower of two bubbles corresponding to solutions of two different Liouville problems in $\mathbb{R}^2$, namely \eqref{LiouvilleEquationINTRO} and \eqref{SingularLiouvilleEquationINTRO}.
\\
In Section \ref{Section:GSymmetric} of this paper we show that the same phenomenon appears in other symmetric domains different from the balls. We obtain this through the asymptotic analysis of the sign-changing solutions found in Theorem \ref{thm:existenceINTRO}.
\\ 
The starting point for this result is an asymptotic analysis of a general family $(u_p)$ of solutions of \eqref{problem1} satisfying the condition \eqref{energylimitINTRO}. This first result, inspired by the paper \cite{Druet} (see also \cite{DruetHebeyRobert}) can be viewed as a first step towards the complete classification of the asymptotic behavior of general sign-changing solutions of \eqref{problem1}.
\\
The hardest part of the proof of this result relies in showing that the rescaling about the minimum point $x_p^-$ converges to a radial singular solution of a singular Liouville problem in $\R^2$. Indeed, while the rescaling of $u_p$ about the maximum point $x_p^+$ can be studied in a ``canonical'' way, the analysis of the rescaling about $x_p^-$ requires additional arguments. In particular the presence of the nodal line, with an unknown geometry, gives difficulties which, obviously, are not present when dealing with positive solutions or with radial sign-changing solutions. Also the proofs of the results for nodal radial solutions of \cite{GrossiGrumiauPacella2} cannot be of any help since they depend strongly on one-dimensional estimates.
\\
We would like to point out that the analysis carried out in \cite{DeMarchisIanniPacellaJEMS} also allows to get the same asymptotic result substituting the bound on the energy with a bound on the Morse index of the solutions, (see \cite{DeMarchisIanniPacellaAMPA}).\\

Finally we observe that the bubble-tower solutions of \eqref{problem} are also interesting in the study of the associated heat flow because they induce a peculiar blow-up phenomenon (see \cite{CazenaveDicksteinWeissler, DicksteinPacellaSciunzi, MarinoPacellaSciunzi} and in particular \cite{DeMarchisIanni}).\\

We conclude by remarking that the phenomenon of nodal solutions of \eqref{problem1}  with positive and negative part concentrating at the same point and having different asymptotic profile does not seem to appear in higher dimension as $p$ approaches the critical Sobolev exponent.\\

Finally nodal solutions to \eqref{problem1} concentrating at a finite number of point without exhibiting the bubble tower phenomenon, i.e. only simple concentration points, also exist (see \cite{EMPnodal}).

\

\

\section{General asymptotic analysis} \label{SectionGeneralAnalysis}

\

This section is mostly devoted to the study of the  asymptotic behavior of a general family $(u_p)_{p>1}$ of nontrivial solutions of \eqref{problem1} satisfying the uniform upper bound  
\begin{equation}
\label{energylimit}
p\int_{\Omega}|\nabla u_p|^2 dx\leq C,\ \mbox{ for some  $C>0$ independent of $p$.}
\end{equation}
At the end of the section we also exhibit some recent results related to families of positive solutions.
The material presented is mainly based on some of the results contained in \cite{DeMarchisIanniPacellaJEMS} plus smaller additions or minor improvements. We also refer to \cite{DIPpositive} for the complete analysis in the case of  positive solutions.

\

Recall that in \cite{RenWeiTAMS1994}  it has been proved that for any family 
$(u_p)_{p>1}$ of nontrivial solutions of \eqref{problem}  the following lower bound holds
\begin{equation}
\label{energylimitLower}
\liminf_{p\rightarrow+\infty}p\int_{\Omega}|\nabla u_p|^2 dx\geq 8\pi e,
\end{equation}
so the constant $C$ in \eqref{energylimit} is intended to satisfy $C\geq 8\pi e$. Moreover if $u_p$ is sign-changing then we also know that (see again \cite{RenWeiTAMS1994})
\begin{equation}
\label{energylimitLowerPosNeg} 
\liminf_{p\rightarrow+\infty}p\int_{\Omega}|\nabla u_p^{\pm}|^2 dx\geq 8\pi e.
\end{equation}

\

If we denote by $E_p$ the energy functional associated to \eqref{problem1}, i.e.
\[
E_p(u):=\frac{1}{2}\|\nabla u\|^2_{2}-\frac{1}{p+1}\|u\|_{p+1}^{p+1},\ \ u\in H^1_0(\Omega),
\]
since for a solution $u$  of \eqref{problem1}
\begin{equation}
\label{energiaSuSoluzioni}
E_p(u)=(\frac12-\frac1{p+1})\|\nabla u\|^2_2=(\frac12-\frac1{p+1})\|u\|^{p+1}_{p+1},
\end{equation}
then \eqref{energylimit}, \eqref{energylimitLower} and \eqref{energylimitLowerPosNeg} are equivalent to uniform upper and lower bounds for the energy $E_p$ or for the $L^{p+1}$-norm, indeed
\begin{eqnarray*}
&
\displaystyle\limsup_{p\rightarrow +\infty}\  2 pE_p(u_p) = \limsup_{p\rightarrow +\infty}\  p\int_{\Omega} |u_p|^{p+1} \, dx = \limsup_{p\rightarrow +\infty}\  p\int_{\Omega} |\nabla u_p|^{2} \, dx\leq C
\\
&\displaystyle \liminf_{p\rightarrow +\infty}\ 2 pE_p(u_p)=\liminf_{p\rightarrow +\infty}\ p\int_{\Omega} |u_p|^{p+1} \, dx = \liminf_{p\rightarrow +\infty}\ p\int_{\Omega} |\nabla u_p|^{2} \, dx\geq 8\pi e
\end{eqnarray*}
and  if $u_p$ is sign-changing, also
\begin{eqnarray*}
&\displaystyle \liminf_{p\rightarrow +\infty}\ 2 pE_p(u_p^{\pm})=\liminf_{p\rightarrow +\infty}\ p\int_{\Omega} |u_p^{\pm}|^{p+1} \, dx = \liminf_{p\rightarrow +\infty}\ p\int_{\Omega} |\nabla u_p^{\pm}|^{2} \, dx\geq 8\pi e, 
\end{eqnarray*}
we will use all these equivalent formulations throughout the paper.
\

\

Observe that by the assumption in \eqref{energylimit} we have that
\[E_p(u_p)\rightarrow 0, \ \ \|\nabla u_p\|_2\rightarrow 0,\ \  \mbox{as $p\rightarrow +\infty$}\]
\[E_p(u_p^{\pm})\rightarrow 0, \ \ \|\nabla u_p^{\pm}\|_2\rightarrow 0,\ \  \mbox{as $p\rightarrow +\infty\qquad$ (if $u_p$ is sign-changing)}\]
so in particular $u_p^{\pm}\rightarrow 0$ a.e. as $p\rightarrow +\infty$.
\\
In this section we will show that the solutions $u_p$ do not vanish as $p\rightarrow +\infty$ (both $u_p^{\pm}$ do not vanish if $u_p$ is sign-changing) and  that moreover, differently with what happens in higher dimension,  they do not blow-up (see Theorem \ref{teo:BoundEnergia} below). Moreover we will show that they concentrate at a finite number of points and we will also describe the asymptotic behavior of suitable rescalings of $u_p$ (``bubbles'') around suitable ``concentrating'' sequences of points (see Theorem \ref{thm:x1N} in the following).

\

\

Our first results is the following:
\begin{theorem}\label{teo:BoundEnergia}
Let $(\upp)$ be a family of solutions to \eqref{problem1} satisfying \eqref{energylimit}. Then
\begin{itemize}

\item[\emph{$(i)$}] (No vanishing).\\
\[\|u_p\|_{\infty}^{p-1}\geq \lambda_1,\] where $\lambda_1=\lambda_1(\Omega)(>0)$ is the first eigenvalue of the operator $-\Delta$ in $H^1_0(\Omega)$.\\
 If $u_p$ is sign-changing then also $\|u_p^\pm\|_{\infty}^{p-1}\geq \lambda_1$.

\item[$(ii)$] (Existence of the first bubble).
Let $(x_p^+)_p\subset\Omega$ such that $|u_p(x_p^+)|=\|u_p\|_{\infty}$.
Let us set
\begin{equation}\label{muppiu}
\mu_{p}^+:=\left(p |\upp(x_p^+)|^{p-1}\right)^{-\frac{1}{2}}
\end{equation}
and  for  $ x\in\widetilde{\Omega}_{p}^+:=\{x\in\R^2\,:\,x_p^++\mu_{p}^+x\in\Omega\}$
\begin{equation}\label{scalingMax}
v_{p}^+(x):=\fr{p}{\upp(x_p^+)}(\upp(x_p^++\mu_{p}^+ x)-\upp(x_{1,p})).
\end{equation}
Then  $\mu_{p}^+\to0$ as $p\to+\infty$  and
\[v_{p}^+\To U\mbox{ in }C^2_{loc}(\R^2)\mbox{ as }p\to+\infty\] where
\begin{equation}\label{v0}
U(x)=\log\left(\fr1{1+\fr18 |x|^2}\right)^2
\end{equation}
is the solution of $-\lap U=e^{U}$ in $\R^2$, $U\leq 0$, $U(0)=0$ and $\int_{\mathbb{R}^2}e^{U}=8\pi$.
\\
Moreover
\begin{equation}\label{soloMaggiore1InGenerale}
\liminf_{p\rightarrow +\infty}\|u_p\|_{\infty}\geq 1.
\end{equation}

\item[\emph{$(iii)$}] (No blow-up). There exists $C>0$ such that 
\begin{equation}
\label{boundSoluzio}\|\upp\|_{\infty}\leq C,\ \mbox{  for all $p>1$.}
\end{equation}
\item[\emph{$(iv)$}] There exist constants $c,C>0$, such that for all $p$ sufficiently large we have 
\begin{equation}
\label{boundEnergiap}c\leq p\int_\Omega |u_p|^p dx \leq C.
\end{equation}
\item[\emph{$(v)$}] $\sqrt{p}u_p\rightharpoonup 0$ in $H^1_{0}(\Omega)$ as $p\rightarrow +\infty$.
\end{itemize}
\end{theorem}
\begin{proof}
Point \emph{$(i)$} has been first proved for positive solutions in \cite{RenWeiTAMS1994}, here we follows the proof in  \cite[Proposition 2.5]{GrossiGrumiauPacella1}. If $u_p$ is sign-changing, just observe that  $u_p^{\pm}\in H^1_0(\Omega)$, where we know that 
\[0<8\pi e-\varepsilon \overset{\eqref{energylimitLower}/\eqref{energylimitLowerPosNeg}}{\leq} \int_{\Omega}|\nabla u_p^{\pm}|^2\,dx\overset{\eqref{energylimit}}{\leq} C<+\infty\] and that also by Poincare inequality
\[\int_{\Omega} |\nabla u_p^{\pm}|^2\,dx=\int_{\Omega}|u_p^{\pm}|^{p+1}\,dx\leq \|u_p^{\pm}\|_{\infty}^{p-1}\int_{\Omega}|u_p^{\pm}|^2\,dx\leq \frac{\|u_p^{\pm}\|_{L^{\infty}(\Omega)}^{p-1}}{\lambda_1(\Omega)}\int_{\Omega}|\nabla u_p^{\pm}|^2\,dx.\]
If $u_p$ is not sign-changing just observe that either $u_p=u_p^+$ or $u_p=u_p^-$ and the same proof as before applies.

\

\

The proof of \emph{$(ii)$} follows the same ideas in \cite{AdiGrossi} where the same result has been proved for least energy (positive) solutions. 
We let $x_{p}^+$ be a point in $\Omega$ where $|\upp|$ achieves its maximum. Without loss of generality we can assume that
\bel\label{x1p}
\upp(x_{p}^+)=\max_\Omega \upp>0.
\eel
By \emph{(i)} we have that 
$
p\upp(x_{p}^+)^{p-1}\to+\infty$ as $p\to+\infty$,
so \eqref{soloMaggiore1InGenerale} holds and moreover
$\mu_{1,p}\to 0$, where $\mu_p^+$ is defined in \eqref{muppiu}. Let
$
\widetilde{\Omega}_{p}^+$ and $
v_{p}^+$ be defined as in  \eqref{scalingMax}, then
by \eqref{x1p} we have
\begin{equation}\label{b}
v_{p}^+(0)=0\quad\textrm{and}\quad v_{p}^+\leq 0\textrm{  in $\widetilde{\Omega}_{p}^+$.}
\end{equation}
Moreover $v_{p}^+$ solves
\bel\label{lapv1p}
-\lap v_{p}^+=\left|1+\frac{v_{p}^+}{p}\right|^p\left(1+\frac{v_{p}^+}{p}\right)\quad\textrm{in $\widetilde{\Omega}_{p}^+$},
\eel
with \[\left|1+\frac{v_{p}^+}{p}\right|\leq 1 \ \ \ \mbox{and} \ \ \ v_{p}^+=-p\quad\textrm{on $\partial\widetilde{\Omega}_{p}^+$.}\]
Then
\bel\label{boundlap1}
|-\lap v_{p}^+|\leq 1\quad\textrm{in $\widetilde{\Omega}_{p}^+$}.
\eel
Using  \eqref{b} and \eqref{boundlap1} we prove that
\begin{equation}\label{tuttoR2}
\widetilde{\Omega}_{p}^+\to\R^2\ \mbox{ as }p\to+\infty.
\end{equation}Indeed since $\mu_{p}^+\rightarrow 0$ as $p\rightarrow + \infty$, either $\widetilde{\Omega}_{p}^+\rightarrow\mathbb R^2$ or $\widetilde{\Omega}_{p}^+\rightarrow\mathbb R\times ]-\infty, R[$ as $p\rightarrow +\infty$ for some $R\geq 0$ (up to a rotation). In the second case we let
\[
v_{p}^+=\varphi_p+\psi_p \ \mbox{ in }\widetilde{\Omega}_{p}^+\cap B_{2R+1}(0)
\]
with
$-\lap\varphi_p=-\lap v_{p}^+$ in $\widetilde{\Omega}_{p}^+\cap B_{2R+1}(0)$ and $\psi_p= v_{p}^+$ in $\partial\left(\widetilde{\Omega}_{p}^+\cap B_{2R+1}(0)\right)$.

Thanks to \eqref{boundlap1}  we have, by standard elliptic theory, that $\varphi_p$ is uniformly bounded in $\widetilde{\Omega}_{1}^+\cap B_{2R+1}(0)$. So the function $\psi_p$ is harmonic in $\widetilde{\Omega}_{p}^+\cap B_{2R+1}(0)$, bounded from above by \eqref{b} and satisfies  $\psi_p=-p\rightarrow -\infty$ on $\partial\widetilde{\Omega}_{p}^+\cap B_{2R+1}(0)$. Since $\partial\widetilde{\Omega}_{p}^+\cap B_{2R+1}(0)\rightarrow (\mathbb{R}\times\{R\})\cap B_{2R+1}(0)$ as $p\rightarrow +\infty$  one easily gets that $\psi_p(0)\rightarrow -\infty$ as $p\rightarrow +\infty$ (if $R=0$ this is trivial, if $R>0$ it follows by Harnack inequality). This is  a contradiction since $\psi_p(0)=-\varphi_p(0)$ and $\varphi_p$ is bounded, hence \eqref{tuttoR2} is proved.
\\

Then for any $R>0$, $B_R(0)\subset\widetilde\Omega_{1,p}$ for $p$ sufficiently large. So $(v_{p}^+)$ is a family of nonpositive functions with uniformly bounded Laplacian in $B_R(0)$ and with $v_{p^+}(0)=0$.

Thus, arguing as before, we write $v_{p}^+=\varphi_p+\psi_p $ where $\varphi_p$ is uniformly bounded in $B_R(0)$  and $\psi_p $ is an harmonic function which is uniformly bounded from above. By Harnack inequality, either $\psi_p $ is uniformly bounded in $B_R(0)$ or it tends to $-\infty$ on each compact set of $B_R(0)$.
The second alternative cannot happen because, by definition, $\psi_p(0)=v_{p}^+(0)-\varphi_p(0) =-\varphi_p(0)\geq -C$.
Hence we obtain that $v_{p}^+$ is uniformly bounded in $B_R(0)$, for all $R>0$. 
By standard elliptic regularity theory one has that $v_{p}^+$ is bounded in $C^{2,\alpha}_{loc}(\R^2).$ 
Thus by Arzela-Ascoli Theorem and a diagonal process on $R\rightarrow +\infty$, after passing to a subsequence
\bel\label{v1pv0}
v_{p}^+\to U\quad\textrm{in $C^2_{loc}(\R^2)$ as $p\to+\infty$},
\eel
with $U\in C^2(\R^2)$, $U\leq0$ and $U(0)=0$. Thanks to \eqref{lapv1p} (on each ball also $1+\frac{v_{p}^+}{p}>0$ for $p$ large) and \eqref{v1pv0} we get
that $U$ is a solution of
$
-\lap U=e^{U}$ in $\R^2$.
Moreover for any $R>0$, by \eqref{RemarkMaxCirca1}, we have
\begin{eqnarray*}
\int_{B_R(0)}e^{U(x)}dx&\stackrel{\eqref{v1pv0}\,+\,\textrm{{Fatou}}}{\leq} &\int_{B_R(0)}\fr{|\upp(x_{p}^++\mu_{p}^+x)|^{p+1}}{\upp(x_{p}^+)^{p+1}}dx+o_p(1)\\
&=&\fr p{\|u_p\|_{\infty}^2}\int_{B_{R\mu_{1,p}}(x_{1,p})}|\upp(y)|^{p+1}dy+o_p(1)\\
&\stackrel{\eqref{soloMaggiore1InGenerale}}{\leq}&\fr p{(1-\ep)^2}\int_{\Omega}|\upp(y)|^{p+1} dy+o_p(1)\stackrel{\eqref{energylimit}}{\leq} C<+\infty,
\end{eqnarray*}
so that $e^{U}\in L^1(\R^2)$. Thus, since $U(0)=0$, by virtue of the classification due to Chen and Li \cite{ChenLi} we obtain \eqref{v0}. Last an easy computation shows that $\int_{\mathbb{R}^2}e^{U}=8\pi$.

\

Point 
\emph{$(iii)$} has been first proved in \cite{RenWeiTAMS1994}, here we write a simpler proof which follows directly  from (ii) by applying Fatou's lemma. An analogous argument can be found in \cite[Lemma 3.1]{AdiGrossi}. Indeed 
\[C\overset{\eqref{energylimit}}{\geq}p\int_{\Omega}|u_p(y)|^{p+1}dy=\|u_p\|_{\infty}^2\int_{\widetilde{\Omega}_{p}^+}\left|1+\frac{v_{p}^+(x)}{p} \right|^{p+1}dx\overset{\mbox{\footnotesize{\emph{$(ii)$}-Fatou}}}{\geq}\|u_p\|_{\infty}^2\int_{\R^2}e^{U(x)}dx=\|u_p\|_{\infty}^2 8\pi\]

\

\

\emph{$(iv)$} follows directly from  \emph{$(iii)$}. Indeed on the one hand
\[
0<C\overset{\eqref{energylimitLower}-\eqref{energiaSuSoluzioni}}{\leq} p\int_\Omega |u_p|^{p+1}\, dx\leq\|u_p\|_{\infty}p\int_\Omega |u_p|^p \, dx\stackrel{\emph{$(iii)$}}{\leq}C p\int_\Omega |u_p|^p \, dx
\]
On the other hand by H\"older inequality 
\[
p\int_\Omega |u_p|^p\, dx \leq |\Omega|^{\frac{1}{p+1}} p\left(\int_\Omega |u_p|^{p+1}\,dx\right)^{\frac{p}{p+1}}\stackrel{\eqref{energylimit}}{\leq}C.
\]

\

\

To prove \emph{$(v)$} we need  \emph{$(iv)$}. Indeed let us note that, since \eqref{energylimit} holds,
 there exists $w\in H^1_0(\Omega)$ such that, up to a subsequence,  $\sqrt{p}u_p \rightharpoonup  w$ in $H^1_0(\Omega)$. We want to show that $w=0$ a.e. in $\Omega$.

Using the equation \eqref{problem1}, for any test function $\varphi\in C^{\infty}_0(\Omega)$, we have
\[
\int_{\Omega}\nabla (\sqrt{p}u_p)\nabla\varphi \,dx = \sqrt{p}\int_{\Omega}|u_p|^{p-1}u_p\varphi\, dx
\leq\frac{\|\varphi\|_{\infty}}{\sqrt{p}} p \int_{\Omega}|u_p|^{p}\,dx
\overset{(iv)}{\leq}\frac{\|\varphi\|_{\infty}}{\sqrt{p}} C
\]
for $p$ large. Hence
\[
\int_{\Omega}\nabla w\nabla\varphi\,dx=0\quad \forall\varphi\in C^{\infty}_0(\Omega),
\]
which implies that
 $w=0$ a.e. in $\Omega$.
\end{proof}

\

\

In order to show our next results we need to introduce some notations. Given a family $(u_p)$ of solutions of \eqref{problem1}
and assuming that there exists  $n\in\N\setminus\{0\}$ families of points $(\xip)$, $i=1,\ldots,n$  in $\Omega$ such that
\begin{equation}
\label{muVaAZero}
p|\upp(\xip)|^{p-1}\to+\infty\ \mbox{ as }\ p\to+\infty,
\end{equation}
we define the parameters $\mip$ by
\bel\label{mip}
\mip^{-2}=p |\upp(\xip)|^{p-1},\ \mbox{ for all }\ i=1,\ldots,n.
\eel
By \eqref{muVaAZero} it is clear that $\mip\to0$ as $p\to+\infty$ and that
\begin{equation}\label{RemarkMaxCirca1}
\liminf_{p\rightarrow +\infty}|u_p(\xip)|\geq 1.
\end{equation}
Then we define the concentration set
\bel\label{S}
\mathcal{S}=\left\{\lim_{p\to+\infty}\xip,\,i=1,\ldots,n\right\}\subset\bar\Omega
\eel
and the function
\bel\label{RNp}
R_{n,p}(x)=\min_{i=1,\ldots,n} |x-\xip|, \ \forall x\in\Omega.
\eel

Finally we introduce the following properties:
\begin{itemize}
\item[$(\mathcal{P}_1^n)$] For any $i,j\in\{1,\ldots,n\}$, $i\neq j$,
\[
\lim_{p\to+\infty}\fr{|\xip-\xjp|}{\mip}=+\infty.
\]
\item[$(\mathcal{P}_2^n)$] For any $i=1,\ldots,n$, for $x\in\widetilde{\Omega}_{i,p}:=\{x\in\R^2\,:\,x_{i,p}+\mu_{i,p}x\in\Omega\}$
\begin{equation}\label{vip}
v_{i,p}(x):=\fr{p}{\upp(\xip)}(\upp(\xip+\mip x)-\upp(\xip))\To U(x)
\end{equation}
in $C^2_{loc}(\R^2)$ as $p\to+\infty$, where $U$ is the same function in \eqref{v0}.
\item[$(\mathcal{P}_3^n)$] There exists $C>0$ such that
\[
p R_{n,p}(x)^2 |\upp(x)|^{p-1}\leq C
\]
for all $p>1$ and all $x\in \Omega$.
\item[$(\mathcal{P}_4^n)$] There exists $C>0$ such that
\[
p R_{n,p}(x) |\nabla \upp(x)|\leq C
\]
for all $p>1$ and all $x\in \Omega$.
\end{itemize}

\

\begin{lemma}\label{lemma:BoundEnergiaBassino}
If there exists  $n\in\N\setminus\{0\}$ such that the properties $(\mathcal{P}_1^n)$ and $(\mathcal{P}_2^n)$ hold for families $(\xip)_{i=1,\ldots,n}$ of points satisfying \eqref{muVaAZero}, then
\[
p\int_\Omega |\na\upp|^2\,dx\geq8\pi\sum_{i=1}^n \alpha_i^2+o_p(1)\ \mbox{ as }p\rightarrow +\infty,
\]
where  $\alpha_i:=\liminf_{p\to+\infty}|\upp(\xip)|\ (\overset{\eqref{RemarkMaxCirca1}}{\geq} 1)$.

\end{lemma}
\begin{proof}
Let us write, for any $R>0$
\begin{eqnarray}\label{appo1}
p\int_{B_{R\mu_{i,p}}(\xip)} |\upp|^{p+1}\,dx&=&\int_{B_R(0)}\fr{|\upp(\xip+\mip y)|^{p+1}}{|\upp(\xip)|^{p-1}}\,dy\nonumber\\
 &=& \upp^2(\xip)\int_{B_R(0)}\left|1+\fr{v_{i,p}(y)}{p}\right|^{p+1}\,dy.
\end{eqnarray}
Thanks to $(\mathcal{P}_2^n)$, we have
\bel\label{appo2}
\|v_{i,p}-U\|_{L^{\infty}(B_R(0))}=o_p(1)\ \mbox{ as }p\rightarrow +\infty.
\eel
Thus by \eqref{appo1}, \eqref{appo2} and Fatou's lemma
\bel\label{appo3}
\liminf_{p\to+\infty}\,\left(p\int_{B_{R\mu_{i,p}}(\xip)}|\upp|^{p+1}\,dx\right)\geq\alpha^2_i\int_{B_R(0)} e^{U}\, dx.
\eel

Moreover by virtue of $(\mathcal{P}_1^n)$ it is not hard to see that $B_{R\mip}(\xip)\cap B_{R\mu_{j,p}}(x_{j,p})=\emptyset$ for all $i\neq j$. Hence, in particular, thanks to \eqref{appo3}
\[
\liminf_{p\to+\infty}\,\left(p\int_\Omega |\upp|^{p+1}\,dx\right)\geq \sum_{i=1}^n \left(\alpha^2_i\int_{B_R(0)} e^{U}\, dx\right).
\]
At last, since this holds for any $R>0$, we get
\[
p\int_\Omega|\na\upp|^2\,dx=p\int_\Omega |\upp|^{p+1}\,dx\geq \sum_{i=1}^n \alpha_i^2 \int_{\R^2}e^{U}\,dx
+o(1)=8\pi\sum_{i=1}^n \alpha_i^2+o(1) \ \mbox{ as }p\rightarrow +\infty.\]
\end{proof}

\

Next result shows that the solutions concentrate at a finite number of points and also establishes the existence of a maximal number of ``bubbles'' 

\begin{theorem}\label{thm:x1N}
Let $(\upp)$ be a family of solutions to \eqref{problem1} and assume that \eqref{energylimit} holds. Then there exist $k\in\N\setminus\{0\}$ and $k$ families of points $(\xip)$ in $\Omega$  $i=1,\ldots, k$ such that, after passing to a sequence, $(\mathcal{P}_1^k)$, $(\mathcal{P}_2^k)$, and $(\mathcal{P}_3^k)$ hold. Moreover $x_{1,p}=x_p^+$ and, given any family of points $x_{k+1,p}$, it is impossible to extract a new sequence from the previous one such that $(\mathcal{P}_1^{k+1})$, $(\mathcal{P}_2^{k+1})$, and $(\mathcal{P}_3^{k+1})$ hold with the sequences $(\xip)$, $i=1,\ldots,k+1$. At last, we have
\begin{equation}\label{pu_va_a_zero}
\sqrt{p}\upp\to 0\quad\textrm{ in $C^2_{loc}(\bar\Omega\setminus\mathcal{S})$ as $p\to+\infty$.}
\end{equation}
Moreover there exists $v\in C^2(\bar\Omega\setminus\mathcal{S})$ such that
\begin{equation}
\label{pu_va_a_funzione}
p\upp\to v\quad\textrm{ in $C^2_{loc}(\bar\Omega\setminus\mathcal{S})$ as $p\to+\infty$,}
\end{equation}
and $(\mathcal{P}_4^k)$ holds.
\end{theorem}

\begin{proof}

This result is mainly contained in \cite{DeMarchisIanniPacellaJEMS}. The proof is inspired by the one  \cite[Proposition 1]{Druet} (see also \cite{SantraWei}), but we have to deal with an extra-difficulty because we allow the solutions $\upp$ to be sign-changing. We divide the proof in several steps and all the claims are up to a subsequence.

\

{\sl STEP 1. We show that there exists a family $(x_{1,p})$ of points in $\Omega$ such that, after passing to a sequence $(\mathcal{P}^1_2)$ holds.\\}

\

We let  $x_p^+$ be a point in $\Omega$ where $|\upp|$ achieves its maximum. The proof then follows taking $x_{1,p}:=x_p^+$  and using the results in Theorem \ref{teo:BoundEnergia}-(ii).

\

\

{\sl STEP 2. We assume that $(\mathcal{P}_1^n)$ and $(\mathcal{P}_2^n)$ hold for some $n\in\N\setminus\{0\}$. Then we show that either $(\mathcal{P}_1^{n+1})$ and $(\mathcal{P}_2^{n+1})$ hold or $(\mathcal{P}_3^n)$ holds, namely there exists $C>0$ such that
$$
p R_{n,p}(x)^2 |\upp(x)|^{p-1}\leq C
$$
for all $x\in\Omega$, with $R_{n,p}$  defined as in \eqref{RNp}.\\
}

\

Let $n\in\N\setminus\{0\}$ and assume that $(\mathcal{P}_1^n)$ and $(\mathcal{P}_2^n)$ hold while
\bel\label{Pknonvale}
\sup_{x\in\Omega}\left(p R_{n,p}(x)^2 |\upp(x)|^{p-1}\right)\to+\infty\quad\textrm{as $p\to+\infty$}.
\eel
We will prove that $(\mathcal{P}_1^{n+1})$ and $(\mathcal{P}_2^{n+1})$ hold.\\

We let $x_{n+1,p}\in\bar\Omega$ be such that
\bel\label{xk+1}
p R_{n,p}(x_{n+1,p})^2 |\upp(x_{n+1,p})|^{p-1}=\sup_{x\in\Omega}\left(p R_{n,p}(x)^2 |\upp(x)|^{p-1}\right).
\eel
Clearly $x_{n+1,p}\in\Omega$ because $\upp=0$ on $\partial \Omega$.
By \eqref{xk+1} and since $\Omega$ is bounded it is clear that
$$
p|\upp(x_{n+1,p})|^{p-1}\to+\infty\quad\textrm{as $p\to+\infty.$}
$$
We claim that
\bel\label{claimxk+1}
\fr{|\xip-x_{n+1,p}|}{\mu_{i,p}}\to+\infty\quad\textrm{as $p\to+\infty$}
\eel
for all $i=1,\ldots,n$ and $\mu_{i,p}$ as in \eqref{mip}. In fact, assuming by contradiction that there exists $i\in\{1,\ldots,n\}$ such that $|\xip-x_{n+1,p}|/\mip\to R$ as $p\to+\infty$ for some $R\geq0$, thanks to $(\mathcal{P}_2^n)$, we get
\[
\lim_{p\to+\infty}p|\xip-x_{n+1,p}|^2 |\upp(x_{n+1,p})|^{p-1}= R^2\left(\fr{1}{1+\fr18R^2}\right)^2<+\infty,
\]
against \eqref{xk+1}. 
Setting
\bel\label{mk+1}
(\mu_{n+1,p})^{-2}:=p |\upp(x_{n+1,p})|^{p-1},
\eel
by \eqref{Pknonvale} and \eqref{xk+1} we deduce that
\bel\label{Rkmk+1}
\fr{R_{n,p}(x_{n+1,p})}{\mu_{n+1,p}}\to+\infty\quad\textrm{as $p\to+\infty$.}
\eel
Then \eqref{mk+1}, \eqref{Rkmk+1} and $(\mathcal{P}_1^n)$ imply that $(\mathcal{P}_1^{n+1})$ holds with the added sequence $(x_{n+1,p})$.\\

Next we show that also $(\mathcal{P}_2^{n+1})$ holds with the added sequence $(x_{n+1,p})$. Let us define the scaled domain
\[
\widetilde{\Omega}_{n+1,p}=\{x\in\R^2\,:\, x_{n+1,p}+\mu_{n+1,p}x\in\Omega\},
\]
and, for $x\in\widetilde{\Omega}_{n+1,p}$, the rescaled functions
\bel\label{vk+1p}
v_{n+1,p}(x)=\fr{p}{\upp(x_{n+1,p})}(\upp(x_{n+1,p}+\mu_{n+1,p}x)-\upp(x_{n+1,p})),
\eel
which, by \eqref{problem1}, satisfy
\bel\label{lapvk+1p1}
-\lap v_{n+1,p}(x)=\fr{|\upp(x_{n+1,p}+\mu_{n+1,p}x)|^{p-1}\upp(x_{n+1,p}+\mu_{n+1,p}x)}{|\upp(x_{n+1,p})|^{p-1}\upp(x_{n+1,p})}\qquad\textrm{in $\widetilde{\Omega}_{n+1,p}$},
\eel
or equivalently
\bel\label{lapvk+1p2}
-\lap v_{n+1,p}(x)=\left|1+\fr{v_{n+1,p}(x)}{p}\right|^{p-1}\left(1+\fr{v_{n+1,p}(x)}{p}\right)\qquad\textrm{in $\widetilde{\Omega}_{n+1,p}$}.
\eel
Fix $R>0$ and let $(z_{p})$ be any point in $\widetilde{\Omega}_{n+1,p}\cap B_R(0)$, whose corresponding points in $\Omega$ is
$$
x_p=x_{n+1,p}+\mu_{n+1,p} z_p.
$$
Thanks to the definition of $x_{n+1,p}$ we have
\begin{equation}\label{r}
p R_{n,p}(x_p)^2|\upp(x_p)|^{p-1}\leq p R_{n,p}(x_{n+1,p})^2|\upp(x_{n+1,p})|^{p-1}.
\end{equation}
Since $|x_p-x_{n+1,p}|\leq R\mu_{n+1,p}$ we have
\begin{eqnarray*}
R_{n,p}(x_p)&\geq&\min_{i=1,\ldots,n}|x_{n+1,p}-\xip|-|x_p-x_{n+1,p}|\\
&\geq&R_{n,p}(x_{n+1,p})-R\mu_{n+1,p}
\end{eqnarray*}
and, analogously,
$$
R_{n,p}(x_p)\leq R_{n,p}(x_{n+1,p})+R\mu_{n+1,p}.
$$
Thus, by \eqref{Rkmk+1} we get
$$
R_{n,p}(x_p)=(1+o(1))R_{n,p}(x_{n+1,p})
$$
and in turn from \eqref{r}
\bel\label{1+o1}
|\upp(x_p)|^{p-1}\leq(1+o(1))|\upp(x_{n+1,p})|^{p-1}.
\eel

In the following we show that for any compact subset $K$ of $\R^2$
\begin{equation}\label{boundLaplaciano}
-1+o(1)\leq-\lap v_{n+1,p}\leq 1+o(1)
\qquad\textrm{in $\widetilde{\Omega}_{n+1,p}\cap K$}
\end{equation}
and
\begin{equation}\label{vNeg}
\limsup_{p\to+\infty}\sup_{\widetilde \Omega_{n+1,p}\cap K}v_{n+1,p}\leq0.
\end{equation}
In order to prove \eqref{boundLaplaciano} and \eqref{vNeg} we will distinguish $2$ cases.
\begin{itemize}
\item[$(i)$] \emph{Assume that
\begin{equation}
\label{ahah}
\frac{\upp(x_p)}{\upp(x_{n+1,p})}=\frac{|\upp(x_p)|}{|\upp(x_{n+1,p})|}.
\end{equation}}
Then by \eqref{1+o1} we get $|\upp(x_p)|^p\leq(1+o(1))|\upp(x_{n+1,p})|^p$ and so by \eqref{lapvk+1p1}
\begin{equation}\label{firstHalfLapl}
(0\leq)-\lap v_{n+1,p}(z_p)\overset{\eqref{ahah}}{=}\frac{|\upp(x_p)|^p}{|\upp(x_{n+1,p})|^p}
\leq 1+o(1).
\end{equation}
Moreover, since \eqref{lapvk+1p2} implies $-\lap v_{n+1,p}(z_p)=e^{v_{n+1,p}(z_p)}+o(1)$, we get
\begin{equation}
\label{arghlll}
\limsup_{p\to+\infty}v_{n+1,p}(z_p)\leq 0.
\end{equation}
\item[$(ii)$] \emph{Assume that
\begin{equation}
\frac{\upp(x_p)}{\upp(x_{n+1,p})}=-\frac{|\upp(x_p)|}{|\upp(x_{n+1,p})|}.
\end{equation}}
Then by the expression of $v_{n+1,p}$ necessarily 
\begin{equation}
\label{arghl}
v_{n+1,p}(z_p)\leq 0,
\end{equation}
moreover by \eqref{1+o1} 
\begin{equation}\label{firstHalfLap2}
0\geq-\lap v_{n+1,p}(z_p)=-\fr{|\upp(x_p)|^{p}}{|\upp(x_{n+1,p})|^p}\geq-1+o(1).
\end{equation}
\end{itemize}
\eqref{firstHalfLapl} and \eqref{firstHalfLap2} imply \eqref{boundLaplaciano}, while \eqref{arghl}, \eqref{arghlll} 
and the arbitrariness of $z_p$ give \eqref{vNeg}.
\\

Using \eqref{boundLaplaciano} and \eqref{vNeg} we can prove, similarly as in the proof of Theorem \ref{teo:BoundEnergia}-(ii), that
\begin{equation}\widetilde{\Omega}_{n+1,p}\rightarrow\mathbb R^2\ \mbox{ as }p\rightarrow +\infty.
\end{equation}
%
%Indeed, since $\mu_{n+1,p}\rightarrow 0$ as $p\rightarrow + \infty$, either $\widetilde{\Omega}_{n+1,p}\rightarrow\mathbb R^2$ or $\widetilde{\Omega}_{n+1,p}\rightarrow\mathbb R\times ]-\infty, R[$ as $p\rightarrow +\infty$ for some $R\geq 0$ (up to a rotation). In the second case we let
%\[
%v_{n+1,p}=\varphi_p+\psi_p \ \mbox{ in }\widetilde{\Omega}_{n+1,p}\cap B_{2R+1}(0)
%\]
%with
%$-\lap\varphi_p=-\lap v_{n+1,p}$ in $\widetilde{\Omega}_{n+1,p}\cap B_{2R+1}(0)$ and $\psi_p= v_{n+1,p}$ in $\partial\left(\widetilde{\Omega}_{n+1,p}\cap B_{2R+1}(0)\right)$.
%
%Thanks to \eqref{boundLaplaciano} we have, by standard elliptic theory, that $\varphi_p$ is uniformly bounded in $\widetilde{\Omega}_{n+1,p}\cap B_{2R+1}(0)$. The function $\psi_p$ is  harmonic  in $\widetilde{\Omega}_{n+1,p}\cap B_{2R+1}(0)$, bounded from above by \eqref{vNeg} and satisfies  $\psi_p=-p\rightarrow -\infty$ on $\partial\widetilde{\Omega}_{n+1,p}\cap B_{2R+1}(0)$. Since $\partial\widetilde{\Omega}_{n+1,p}\cap B_{2R+1}(0)\rightarrow (\mathbb{R}\times\{R\})\cap B_{2R+1}(0)$ as $p\rightarrow +\infty$   one easily gets that $\psi_p(0)\rightarrow -\infty$ as $p\rightarrow +\infty$ (if $R=0$ this is trivial, if $R>0$ it follows by Harnack inequality). This is  a contradiction since $\psi_p(0)=-\varphi_p(0)$ and $\varphi_p$ is bounded. Therefore the limit domain of $\widetilde{\Omega}_{n+1,p}$ is the whole $\R^2$.
%\\
Then for any $R>0$, $B_R(0)\subset\widetilde{\Omega}_{n+1,p}$ for $p$ large enough and $v_{n+1,p}$ are  functions with uniformly bounded laplacian in $B_R(0)$ and with $v_{n+1,p}(0)=0$. So, by Harnack inequality, $v_{n+1,p}$ is uniformly bounded in $B_R(0)$ for all $R>0$ and then by standard elliptic regularity $v_{n+1,p}\to U$ in $C^2_{loc}(\R^2)$ as $p\to+\infty$ with $U\in C^2(\R^2)$,  $U(0)=0$ and, by \eqref{vNeg}, $U\leq0$. Passing to the limit in \eqref{lapvk+1p2} we get that $U$ is a solution of  $-\lap U=e^{U}$ in $\R^2$. hen by Fatou's Lemma, as in the proof of Theorem \ref{teo:BoundEnergia}-(ii), we get that $e^{U}\in L^1(\R^2)$ and so by the classification result in \cite{ChenLi} we have the explicit expression of $U$.

This proves that $(\mathcal{P}_2^{n+1})$ holds with the added points $(x_{n+1,p})$, thus {\sl STEP 2.} is proved.

\

\

{\sl STEP 3. We complete the proof of Theorem \ref{thm:x1N}.}

\

From {\sl STEP 1.} we have that $(\mathcal{P}_1^1)$ and $(\mathcal{P}_2^1)$ hold. Then, by {\sl STEP 2.}, either $(\mathcal{P}_1^2)$ and $(\mathcal{P}_2^2)$ hold or
$(\mathcal{P}_3^1)$ holds. In the last case the assertion is proved with $k=1$. In the first case we go on and proceed with the same alternative until we reach a number $k\in\N\setminus\{0\}$ for which $(\mathcal{P}_1^{k})$, $(\mathcal{P}_2^{k})$ and $(\mathcal{P}_3^{k})$ hold up to a sequence. Note that this is possible because the solutions $u_p$ satisfy  \eqref{energylimit} and Lemma \ref{lemma:BoundEnergiaBassino} holds and hence the maximal number $k$ of families of points for which
$(\mathcal{P}_1^{k})$, $(\mathcal{P}_2^{k})$ hold must be finite.

\

Moreover, given any other family of points $x_{k+1,p}$, it is impossible to extract a new sequence from it such that $(\mathcal{P}_1^{k+1})$, $(\mathcal{P}_2^{k+1})$ and $(\mathcal{P}_3^{k+1})$ hold together with the points $(x_{i,p})_{i=1,..,k+1}$. Indeed if $(\mathcal{P}_1^{k+1})$ was verified then
\[\frac{|x_{k+1,p}-\xip|}{\mu_{k+1,p}}\to+\infty\quad \mbox{ as } p\to+\infty,\ \mbox{ for any }i\in\{1,\ldots,k\},\]
but this would contradict $(\mathcal{P}_3^k)$.

\

Finally the proofs of \eqref{pu_va_a_zero} and \eqref{pu_va_a_funzione} are a direct consequence of $(\mathcal{P}^k_3)$.
Indeed if $K$ is a compact subset of $\bar\Omega\setminus\mathcal S$ by $(\mathcal{P}^k_3)$ we have that there exists $C_K>0$ such that
$$
p|\upp(x)|^{p-1}\leq C_K\qquad\textrm{for all $x\in K$.}
$$
Then by  \eqref{problem1} $\|\lap(\sqrt{p} \upp)\|_{L^\infty(K)}\leq C_K\frac{\|u_p\|_{L^\infty(K)}}{\sqrt{p}}\to0$, as $p\to+\infty$. Hence standard elliptic theory shows that
 $\sqrt{p}u_p\to w$ in $C^2(K)$, for some $w$. But by  Theorem \ref{teo:BoundEnergia} we know that
$\sqrt{p}u_p\rightharpoonup 0$, so $w=0$ and \eqref{pu_va_a_zero} is proved. 
Iterating we then have
$\|\lap(p \upp)\|_{L^\infty(K)}\leq C_K\|u_p\|_{L^\infty(K)}\to0$, as $p\to+\infty$ by \eqref{pu_va_a_zero}.
And so by standard elliptic theory we have that $pu_p\rightarrow v$ in $C^2(K)$, for some $v$. The arbitrariness of the compact set $K$  
ends the proof of \eqref{pu_va_a_funzione}.

\

It remains to prove $(\mathcal{P}_4^k)$. Green's representation gives
\begin{eqnarray}
\label{gradientePrima}
\nonumber
p|\nabla u_p(x)|&=& p\left|\int_\Omega \nabla G(x,y) u_p(y)^p dy\right|\leq p\int_\Omega |\nabla G(x,y)| |u_p(y)|^p dy
\\
&\leq & C p \int_\Omega \frac{|\upp(y)|^p}{|x-y|} dy,
\end{eqnarray}
where $G$ is the Green function of $-\Delta$ in $\Omega$ with Dirichlet boundary conditions, and in the last estimate we have used that $|\nabla_x G(x,y)|\leq \frac{C}{|x-y|}\quad\forall x,y\in\Omega,\ x\neq y$ (see for instance \cite{DallAcquaSweers}).
Let $R_{k,p}(x)=\min_{i=1,\ldots,k}|x-\xip|$ and $\Omega_{i,p}=\{x\in\Omega\,:\,|x-\xip|=R_{i,p}(x)\}$, $i=1,\ldots,k$. We then have 
\begin{eqnarray*}
p\,\int_{\Omega_{i,p}}|x-y|^{-1}|\upp(y)|^p\,dy&=&p\,\int_{\Omega_{i,p}\cap B_{\frac{|x-x_{i,p}|}{2}}(x_{i,p})}|x-y|^{-1}|\upp(y)|^p\,dy\\
&&+\ p\,\int_{\Omega_{i,p}\setminus B_{\frac{|x-x_{i,p}|}{2}}(x_{i,p})}|x-y|^{-1}|\upp
(y)|^p\,dy.
\end{eqnarray*}
Note that by \eqref{boundSoluzio} and $(\mathcal P^k_3)$ for $y\in\Omega_{i,p}\setminus B_{\frac{|x-x_{i,p}|}{2}}(x_{i,p})$ we have 
\[
\frac{p\,|\upp(y)|^p}{|x-y|}\leq C\,\frac{p\,|\upp(y)|^{p-1}}{|x-y|}\leq\frac{C}{|x-y||y-x_{i,p}|^2}\leq\frac{C}{|x-y||x-x_{i,p}|^2},
\]
and hence
\begin{eqnarray*}
\int_{\Omega_{i,p}\setminus B_{\frac{|x-x_{i,p}|}{2}}(x_{i,p})}\frac{p\,|\upp(y)|^p}{|x-y|}\,dy&\leq&\frac{1}{|x-x_{i,p}|^2}\int_{|x-y|\leq |x-x_{i,p}|}\frac{C}{|x-y|}\,dy\,+\,\frac{1}{|x-x_{i,p}|}\int_{\Omega_{i,p}}p\,|\upp(y)|^p dy\\
& \overset{\eqref{boundEnergiap}}{\leq}& \frac{C}{|x-x_{i,p}|}.
\end{eqnarray*}
For $y\in \Omega_{i,p}\cap B_{\frac{|x-x_{i,p}|}{2}}(x_{i,p})$, $|x-y|\geq|x-\xip|-|y-\xip|\geq|x-\xip|/2$, and hence by \eqref{boundSoluzio} and \eqref{boundEnergiap} we get
\[
p\,\int_{\Omega_{i,p}\cap B_{\frac{|x-x_{i,p}|}{2}}(x_{i,p})}\frac{|\upp(y)|^p}{|x-y|}\,dy\leq\frac{C}{|x-\xip|},\quad\mbox{for $i=1,\ldots,k$}
\]
so that $(\mathcal{P}_4^k)$ is proved.
\end{proof}

\

\

In the rest of this section we derive some consequences of Theorem \ref{thm:x1N}.

\

\begin{remark}\label{rem:nonvedobordo}
Under the assumptions of Theorem \ref{thm:x1N} we have
$$
\fr{dist(\xip,\partial\Omega)}{\mip}\underset{p\to+\infty}{\longrightarrow}+\infty\qquad\textrm{for all $i\in\{1,\ldots,k\}$}.
$$
\end{remark}
\begin{corollary}\label{cor:nonvedoNL}
Under the assumptions of Theorem \ref{thm:x1N} if $u_p$ is sign-changing it follows that
$$
\fr{dist(\xip,NL_p)}{\mip}\underset{p\to+\infty}{\longrightarrow}+\infty\quad\textrm{ for all $i\in\{1,\ldots,k\}$}
$$
where $NL_p$ denotes the  nodal line of $\upp$.\\
As a consequence, for any $i\in\{1,\ldots,k\}$, letting $\mathcal{N}_{i,p}\subset\Omega$ be the nodal domain of $u_p$ containing $x_{i,p}$ and setting $u_p^i:=u_p\chi_{\mathcal{N}_{i,p}}$ ($\chi_A$ is the characteristic function of the set $A$),
then the scaling of $u_p^i$ around $x_{i,p}$:
\[
z_{i,p}(x):=\fr{p}{\upp(\xip)}(\upp^i(\xip+\mip x)-\upp(\xip)),
\]
defined on $\widetilde{\mathcal{N}}_{i,p}:=\frac{\mathcal{N}_{i,p}-x_{i,p}}{\mu_{i,p}}$, converges to $U$ in $C^2_{loc}(\mathbb R^2)$, where $U$ is the same function defined in $(\mathcal{P}_2^k)$.
\end{corollary}
\begin{proof}
Let us suppose by contradiction that
$$
\fr{dist(\xip,NL_p)}{\mip}\underset{p\to+\infty}{\longrightarrow}\ell\geq0,
$$
then there exist $y_p\in NL_p$ such that $\fr{|\xip-y_p|}{\mip}\to\ell$ as $p\rightarrow +\infty$.
Setting
$$
v_{i,p}(x):=\fr{p}{\upp(\xip)}(\upp(\xip+\mip x)-\upp(\xip)),
$$
on the one hand
$$
v_{i,p}(\fr{y_p-\xip}{\mip})=-p\underset{p\to+\infty}{\longrightarrow}-\infty,
$$
on the other hand by $(\mathcal{P}_2^k)$ and up to subsequences
$$
v_{i,p}(\fr{y_p-\xip}{\mip})\underset{p\to+\infty}{\longrightarrow}U(x_{\infty})>-\infty,
$$
where $x_\infty=\lim_{p\rightarrow +\infty} \fr{y_p-\xip}{\mip}\in\R^2$ and so $|x_\infty|=\ell$. Thus we have obtained a contradiction which proves the assertion.
\end{proof}

\

For a family of points $(x_p)_p\subset\Omega$ we denote by $\mu(x_p)$ the numbers defined by
\begin{equation} 
\label{defmup}
\left[\mu (x_p)\right]^{-2}:=p |u_p(x_p)|^{p-1}.
\end{equation}

\begin{proposition}\label{lemma:rapportoMuBounded} Let $(x_p)_p\subset\Omega$ be a family of points such that  $p |u_p(x_p)|^{p-1}\rightarrow +\infty$ and let $\mu (x_p)$ be as in \eqref{defmup}.
By $(\mathcal{P}_3^{k})$, up to a sequence, $R_{k,p}(x_p)=|x_{i,p}-x_p|$, for a certain  $i\in\{1,\dots,k\}$. Then
\[
\limsup_{p\rightarrow +\infty}\frac{\mu_{i,p}}{\mu (x_p)}\leq  1.
\]
\end{proposition}
\begin{proof}
To shorten the notation let us denote $\mu (x_p)$ simply by $\mu_p$. Let us start by proving that $\frac{\mu_{i,p}}{\mu_p}$ is bounded. So we assume by contradiction that there exists a sequence $p_n\rightarrow +\infty$, as  $n\rightarrow +\infty$ such that
\begin{equation}\label{ipotAssurdo2}
\frac{\mu_{i,p_n}}{\mu_{p_n}}\rightarrow +\infty \quad \mbox{ as } n\rightarrow +\infty.
\end{equation}
By $(\mathcal{P}_3^{k})$ and \eqref{ipotAssurdo2} we then have
\[
\frac{|x_{p_n}-x_{i,p_n}|}{\mu_{i, p_n}}=\frac{|x_{p_n}-x_{i,p_n}|}{\mu_{ p_n}}\frac{\mu_{p_n}}{\mu_{i,p_n}}\rightarrow 0 \quad \mbox{ as } n\rightarrow +\infty,
\]
so that by $(\mathcal{P}_2^{k})$
\[
v_{i,p_n}\left(\frac{x_{p_n}-x_{i,p_n}}{\mu_{i, p_n}}\right)\rightarrow U(0)= 0\quad \mbox{ as } n\rightarrow +\infty.
\]
As a consequence
\[
\frac{\mu_{i,p_n}}{\mu_{p_n}}=\left(\frac{u_{p_n}(x_{p_n})}{u_{p_n}(x_{i,p_n})}\right)^{p_n-1}=\left(1+\frac{v_{i,p_n}\left(\frac{x_{p_n}-x_{i,p_n}}{\mu_{i,p_n}} \right)}{p_n}\right)^{p_n-1}\rightarrow e^{U(0)}=1 \quad \mbox{ as } n\rightarrow +\infty,
\]
which contradicts with \eqref{ipotAssurdo2}. Hence we have proved that $\frac{\mu_{i,p}}{\mu_p}$ is bounded.
\\
Next we show that $\tfrac\mip{\mu_p}\leq 1.$ Assume by contradiction that there exists $\ell >1$ and a sequence $p_n\rightarrow +\infty$ as $n\rightarrow +\infty$ such that
\begin{equation}\label{ipotAssurdo3}
\frac{\mu_{i,p_n}}{\mu_{p_n}}\rightarrow \ell \quad \mbox{ as } n\rightarrow +\infty.
\end{equation}
By $(\mathcal{P}_3^{k})$ and \eqref{ipotAssurdo3} we then have
\[
\frac{|x_{p_n}-x_{i,p_n}|}{\mu_{i, p_n}}=\frac{|x_{p_n}-x_{i,p_n}|}{\mu_{ p_n}}\frac{\mu_{p_n}}{\mu_{i,p_n}}\leq \frac{2\sqrt C}{\ell}
\]
for $n$ large, so that by $(\mathcal{P}_2^{k})$ there exists $x_\infty\in\mathbb R^2$, $|x_\infty|\leq \frac{2\sqrt C}{\ell}$ such that, up to a subsequence
\[
v_{i,p_n}\left(\frac{x_{p_n}-x_{i,p_n}}{\mu_{i, p_n}}\right)\rightarrow U(x_\infty)\leq 0 \quad \mbox{ as } n\rightarrow +\infty.
\]
As a consequence
\[
\frac{\mu_{i,p_n}}{\mu_{p_n}}=\left(\frac{u_{p_n}(x_{p_n})}{u_{p_n}(x_{i,p_n})}\right)^{p_n-1}=\left(1+\frac{v_{i,p_n}\left(\frac{x_{p_n}-x_{i,p_n}}{\mu_{i,p_n}} \right)}{p_n}\right)^{p_n-1}\rightarrow e^{U(x_\infty)}\quad \mbox{ as } n\rightarrow +\infty.
\]
By \eqref{ipotAssurdo3} and the assumption $\ell >1$ we deduce
\[U(x_{\infty})=\log \ell +o_n(1)>0\]
reaching a contradiction.
\end{proof}

\begin{proposition}\label{lemma:rappMuZero}
Let $x_p$ and $x_{i,p}$ be as in the statement of Proposition \ref{lemma:rapportoMuBounded}.
If \begin{equation}\label{ConditionNonVedo}\tfrac{|x_p-\xip|}{\mip}\to+\infty \quad \mbox{ as } p\rightarrow +\infty,\end{equation}
then \[\tfrac{\mip}{\mu (x_p)}\to 0 \quad \mbox{ as } p\rightarrow +\infty,\]
where $\mu (x_p)$ is defined in \eqref{defmup}.
\end{proposition}
\begin{proof}
By Proposition \ref{lemma:rapportoMuBounded} we know that
\[
\frac{\mu_{i,p}}{\mu (x_{p})}\leq 1 +o(1).
\]
Assume by contradiction that \eqref{ConditionNonVedo} holds but there exists $0<\ell \leq 1$ and a sequence $p_n\rightarrow +\infty$ such that
\begin{equation}\label{ipotAssurdo4}
\frac{\mu_{i,p_n}}{\mu (x_{p_n})}\rightarrow \ell, \ \mbox{ as } n\rightarrow +\infty.
\end{equation}
Then \eqref{ipotAssurdo4} and $(\mathcal{P}_3^{k})$ imply
\[
\frac{|x_{p_n}-x_{i,p_n}|}{\mu_{i, p_n}}=\frac{|x_{p_n}-x_{i,p_n}|}{\ell\  \mu (x_{p_n})}+o_{n}(1)\leq \frac{C}{\ell} +o_n(1) \ \mbox{ as }n\rightarrow +\infty
\]
which contradicts \eqref{ConditionNonVedo}.
\end{proof}

\

\begin{remark} \label{remarkSegnoNegativoMu}
If  $u_p(x_p)$ and $u_p(x_{i,p})$ have opposite sign, i.e.
\[u_p(x_p)u_p(x_{i,p})<0,\]
then, by Corollary \ref{cor:nonvedoNL}, necessarily \eqref{ConditionNonVedo} holds.
Hence in this case \[\tfrac{\mip}{\mu (x_p)}\to 0 \ \mbox{ as } p\rightarrow +\infty.\]
\end{remark}

\

\

Next result characterizes in different ways the concentration set $\mathcal S$.

\begin{proposition}[Characterizations of $\mathcal S$]\label{prop:stime gradiente}\label{propozition: caratterizzazioneS}
Let $(u_p)$ be a family of solutions to \eqref{problem1} satisfying \eqref{energylimit}. Then the following holds:  
\begin{equation}\label{carattSi}\mathcal S=\left\{x\in\overline\Omega\, :\, \forall\, r_0>0,\  \forall\, p_0>1,\ \exists\, p>p_0\  \mbox{ s.t. }\ p\int_{B_{r_0}(x)\cap\Omega} |u_p(x)|^{p+1}\ dx\geq 1\right\};
\end{equation}
\begin{equation}\label{carattSii}
\mathcal{S}=\left\{x\in\overline\Omega\, :\, \exists\, \mbox{a subsequence of $(\upp)$ and a sequence $x_p\to_p x$ s.t. $p|\upp(x_p)|\to_p+\infty$} \right\}.
\end{equation}
\end{proposition}
\begin{proof}
Proof of \eqref{carattSi}: by $(\mathcal{P}^k_3)$ it is immediate to see that if $x\notin\mathcal{S}$ then $p\int_{B_r(x)\cap\Omega}|\upp(x)|^{p+1}\,dx\to0$ as $r\to0^+$, uniformly in $p$. On the other hand if $x\in\mathcal S$, i.e. $x=\lim_{p\to+\infty}\xip$ for some $i=1,\ldots,k$, we fix $R>0$ such that $\int_{B_R(0)}e^U\,dx>1$ (where $U$ is defined in \eqref{v0}) and then for $p$ large, reasoning as in the proof of Lemma \ref{lemma:BoundEnergiaBassino}, we get:
\[
p\int_{B_{r_0}(x)\cap\Omega}|\upp(x)|^{p+1}dx\geq p\int_{B_{R\mu_{i,p}}(\xip)}|\upp(x)|^{p+1}\,dx  = |u_p(x_{i,p})|^2\int_{B_R(0)} \left(1+\frac{v_{j,p}(y)}{p}\right)^{p+1}\, dy,
\]
where by Fatou's lemma
\begin{eqnarray*}
\liminf_{p\to+\infty} |u_p(x_{i,p})|^2\int_{B_R(0)} \left(1+\frac{v_{j,p}(y)}{p}\right)^{p+1}\, dy&\geq& \liminf_{p\to+\infty}|\upp(\xip)|^2\int_{B_R(0)}e^{U(y)}\,dy
\\
& \overset{\eqref{RemarkMaxCirca1}}{\geq}& \int_{B_R(0)}e^{U(y)}\,dy>1.
\end{eqnarray*}

Proof of \eqref{carattSii}: if $x\notin \mathcal S$, then by \emph{$(\mathcal{P}_4^k)$}, which holds by Theorem \ref{thm:x1N}, $p|\upp|$ is uniformly bounded in $L^{\infty}(K)$ for some compact set $K$ containing $x$ and then it can not exist a sequence $x_p\to x$ such that $p|\upp(x_p)|\to+\infty$. Conversely if $x\in\mathcal S$, i.e. $x=\lim_{p\to+\infty}\xip$ for some $i=1,\ldots,k$, and by \eqref{muVaAZero} we know that $|\upp(x_{i,p})|\geq\frac12$ for $p$ large, therefore $p|\upp(\xip)|\to+\infty$. This proves \eqref{carattSii}.
\end{proof}

\

\

We conclude this section with a result for positive solutions that we have recently obtained  (\cite{DIPpositive}) carrying on the asymptotic analysis started in Theorem \ref{thm:x1N}.

\begin{theorem}
\label{teo:Positive} 
Let $(u_p)$ be a family of positive solutions to \eqref{problem1} satisfying
\begin{equation}
\label{energylimitCappositive}
p\int_{\Omega}|\nabla u_p|^2 dx\to C,\ \mbox{ as $p\to+\infty$ } C \geq 8\pi e.
\end{equation}
Let $k\in \N\setminus\{0\}$ and $(x_{i,p})$, $i=1,\ldots,k$, the integer and the families of points of $\Omega$ defined in Theorem \ref{thm:x1N}. If we denote by $x_i=\lim\limits_{p\to+\infty}\xip$, then, up to a subsequence we have:
\begin{itemize}
\item[$(i)$] $x_i\in\Omega$,  $x_i\neq x_j$ for $i\neq j$, therefore the concentration set $\mathcal S$, introduced in \eqref{S}, consists of $k$ points
\[
\mathcal S=\{x_1,\ldots,x_k\}\subset \Omega;
\]
\item[$(ii)$]  
\[
pu_p(x)\to  8\pi \sum_{i=1}^k m_i G(x,x_i)\ \mbox{as} \ p\rightarrow +\infty,\  \mbox{in}\ 
C^2_{loc}(\bar\Omega\setminus\mathcal S),
\]
where $m_i:=\lim_{p\rightarrow +\infty} \|u_p\|_{L^{\infty}(\overline{B_{\delta}(x_i)})}$, for any $\delta>0$ such that $B_\delta(x_i)$ does not contain any other $x_j$, $j\neq i$, and $G$ is the Green's function of $-\Delta$ in $\Omega$ under Dirichlet boundary conditions;
\item[$(iii)$] 
\[p\int_\Omega |\nabla u_p(x)|^2\,dx\to 8\pi \sum_{i=1}^k m_i^2,\  \mbox{ as }p\to+\infty;\]
\item[$(iv)$] the concentration points $x_i, \ i=1,\ldots, k$ satisfy  
\begin{equation}\label{x_j relazione}
m_i \nabla_x H(x_i,x_i)+\sum_{i\neq \ell}m_\ell\nabla_x G(x_i,x_\ell)=0,
\end{equation}
where
\begin{equation}\label{H parteregolareGreen}
H(x,y)=G(x,y)+\frac{\log(|x-y|)}{2\pi}
\end{equation}
is the regular part of the Green's function $G$.
\item[$(v)$]
\[m_i \geq \sqrt{e},\qquad \forall i=1,\ldots,k.\]
So in particular \[\lim_{p\rightarrow +\infty} \|u_p\|_{\infty}\geq \sqrt{e}\]
and we have an estimate of $k$ in terms of the constant $C$ in the assumption \eqref{energylimitCappositive}:
\[
C\geq k\,8\pi e.\]
\end{itemize}
\end{theorem}

\

\begin{remark}
For least energy solutions $p\int_{\Omega}|\nabla u_p|^2\rightarrow 8\pi e$ (see \cite{RenWeiTAMS1994} and \cite{RenWeiPAMS1996}) and so Theorem \ref{teo:Positive} implies that $k=1$ and  $\lim\limits_{p\to+\infty}\|u_p\|_{\infty}=\sqrt{e}$, which was already known from \cite{AdiGrossi}.
\end{remark}

\

\

\section{$G$-symmetric case} \label{Section:GSymmetric}

\

In this section we focus on sign-changing solutions. Of course all the results in Section \ref{SectionGeneralAnalysis} hold true if assumption \eqref{energylimit} is satisfied, in particular Theorem \ref{teo:BoundEnergia} and Theorem \ref{thm:x1N}.
\\

Hence letting  $x_p^{\pm}$ be the family of points where $|u_p(x_p^{\pm})|=\|u_p^{\pm}\|_{\infty}$, then from Theorem \ref{teo:BoundEnergia}-\emph{$(i)$} we know that for $x_p^{\pm}$ the analogous of \eqref{muVaAZero} and \eqref{RemarkMaxCirca1} hold and so we have
\begin{equation}\label{mumenozero}
\mu_{p}^{\pm}:=\left(p |\upp(x_p^{\pm})|^{p-1}\right)^{-\frac{1}{2}}\rightarrow 0\ \mbox{ as }p\rightarrow +\infty.
\end{equation}

From now on w.l.g. we assume that the $L^{\infty}$-norm of $u_p$ is assumed at a maximum point, namely that $u_p(x_p^+)=\|u_p^+\|_{\infty}=\|u_p\|_{\infty}$ and that $-u_p(x_p^-)=\|u_p^-\|_{\infty}$.\\

So by Theorem \ref{teo:BoundEnergia}-(\emph{$ii$}) we already know that scaling  $u_p$ about the maximum point $x_p^+$ as in \eqref{scalingMax} gives a first ``bubble'' converging to the function $U$ defined in \eqref{v0}.\\

In general, for sign-changing solutions,  one would like to investigate the behavior of $u_p$ when scaling about the minima $x_p^-$ and understand whether $x_p^-$  coincides with one of the $k$ sequences in Theorem \ref{thm:x1N} or not. Moreover one would like to describe the set of concentration $\mathcal S$.
\\

Recall that if  $x_p^-$ is one of the sequences of Theorem \ref{thm:x1N} then by Corollary \ref{cor:nonvedoNL} one has  
\begin{equation}\label{condizioneN}
\frac{dist(x_p^-,NL_p)}{\mu_p^-}\rightarrow +\infty,\ \mbox{ as }\ p\rightarrow +\infty,
\end{equation}
where $NL_p$ denotes the nodal line of $u_p$. On the contrary it is easy to see that if \eqref{condizioneN} is satisfied, then
one can use the same ideas as in the proof of Theorem  \ref{teo:BoundEnergia}-(\emph{$ii$}) also for the scaling about the minimum which we define in the natural way as
\begin{equation}\label{riscalataAboutMin}
v_p^-(x):=p\frac{u_p(x_p^-+\mu_p^- x)-u_p(x_p^-)}{u_p(x_p^-)}, \quad x\in \widetilde{\Omega}_p^-  :=\{x\in\R^2\,:\,x_p^-+\mu_{p}^-x\in\Omega\},
\end{equation}
and obtain that
$v_p^-\rightarrow U$ in $C^2_{loc}(\R^2)$ (this has been done for instance in \cite{GrossiGrumiauPacella1} for the case of low-energy sign-changing solutions under some additional assumptions).\\

Here we  analyze the case when $u_p$ belongs to a family of $G$-symmetric sign-changing solutions satisfying the same properties as the ones in Theorem \ref{thm:existenceINTRO} recalled in the Introduction and show that  a different phenomenon appears.

All the results of this section are mainly based on the work \cite{DeMarchisIanniPacellaJEMS}, the existence result (Theorem \ref{thm:existenceINTRO}) is instead contained in \cite{DeMarchisIanniPacellaJDE}.\\

We prove the following:

\begin{theorem}
\label{TeoremaPrincipaleCasoSimmetrico}
Let $\Omega\subset\R^2$ be a connected bounded smooth domain, invariant under the action of a cyclic group $G$ of
rotations about the origin,
with $|G|\geq 4e$ ($|G|$ is the order of $G$) and such that the origin $O\in\Omega$.
Let $(\upp)$ be a family of sign-changing $G$-symmetric  solutions of \eqref{problem1} with two nodal regions, $NL_p\cap\partial\Omega=\emptyset$, $O\not\in NL_p$ and
\bel\label{assumptionEnergyINIZIO}
p\int_{\Omega}|\nabla\upp|^2 dx\leq\alpha\,8\pi e
\eel
for some  $\alpha<5$ and $p$ large. Then, assuming w.l.o.g. that $\|u_p\|_{\infty}=\|u_p^+\|_{\infty}$, we have:
\begin{itemize}
\item[i)] $\mathcal S=\{O\}$ and $k=1$\\
where $S$ and $k$ are the ones in  Theorem \ref{thm:x1N};
\item[ii)] $|x_p^{+}|\rightarrow O \ \mbox{ as }\ p\rightarrow +\infty$;
\item[iii)] $|x_p^{-}|\rightarrow O \ \mbox{ as }\ p\rightarrow +\infty$;
\item[iv)] $NL_p$ shrinks to the origin as $p\rightarrow +\infty$;
\item[v)] There exists $x_{\infty}\in\mathbb R^2\setminus\{ 0\}$ such that, up to a subsequence, $\frac{x_p^-}{\mu_p^-}\rightarrow - x_{\infty} $ and
\[v_p^-(x)\longrightarrow V(x-x_{\infty})\quad\mbox{  in
}C^2_{loc}(\R^2\setminus\{x_{\infty}\})\mbox{ as }p\rightarrow +\infty,\] where  
\begin{equation}
\label{espressioneEsplicitaV}
V(x):=\log\left(\frac{2\alpha^2\beta^{\alpha}|x|^{\alpha -2}}{(\beta^{\alpha}+|x|^{\alpha})^2} \right),
\end{equation}
with $\alpha=\alpha(|x_{\infty}|)=\sqrt{2|x_{\infty}|^2+4}$, $\beta=\beta(|x_{\infty}|)=|x_{\infty}| \left(\frac{\alpha+2}{\alpha-2} \right)^{\frac{1}{\alpha}}$, is
a singular radial solution  of 
\begin{equation}
\label{LiouvilleSingularEquation}
\left\{
\begin{array}{lr}
-\Delta V=e^V+ H\delta_{0}\quad\mbox{ in }\R^2\\
\int_{\R^2}e^Vdx<\infty
\end{array}
\right.
\end{equation}
where $H=H(|x_{\infty}|)<0$ is a suitable constant and $\delta_{0}$ is the Dirac measure centered at $0$.
\end{itemize} 
\end{theorem}

Observe that the existence of families of solutions $u_p$ having all the properties as in the assumptions of Theorem \ref{TeoremaPrincipaleCasoSimmetrico} has been proved in \cite{DeMarchisIanniPacellaJDE} when $\Omega$ is simply connected and when $|G|>4$ (see Theorem \ref{thm:existenceINTRO} in the Section \ref{SectionClassicalResults}).\\

We also recall that in \cite{GrossiGrumiauPacella2} the case of least energy  sign-changing radial solutions in a ball has been studied, proving a result similar to that in Theorem \ref{TeoremaPrincipaleCasoSimmetrico} with precise estimates of $\alpha,\beta$ and $H$.

\

\subsection{Proofs of \emph{$(i)-(ii)$} of Theorem \ref{TeoremaPrincipaleCasoSimmetrico}}\label{Subsection:GSymmetricMax}

\

Let us introduce the following notations:
\begin{itemize}
\item $\mathcal{N}_p^{\pm}\subset\Omega$ denotes the positive/negative nodal domain of $u_p$

\item $\widetilde{\mathcal{N}}_p^{\pm}$ are the rescaled nodal domains about the points $x_p^{\pm}$ by the parameters $\mu_p^{\pm}$ defined in the introduction, i.e. \[\widetilde{\mathcal{N}}_p^{\pm}:=\frac{\mathcal{N}_p^{\pm}-x_p^{\pm}}{\mu_p^{\pm}}=\{x\in\mathbb R^2: x_p^{\pm}+\mu_p^{\pm}x\in \mathcal{N}_p^{\pm}\}.\]
\end{itemize}

Let $k$, $(x_{i,p})$, $i=1,\ldots,k$ and $\mathcal{S}$ be as in
 Theorem \ref{thm:x1N} then, defining $\mu_{i,p}$ as in \eqref{mip}, we get

\begin{proposition}\label{MaxVaazero} Under the assumptions of Theorem \ref{TeoremaPrincipaleCasoSimmetrico}. 
$$\frac{|x_{i,p}|}{\mu_{i,p}}\mbox{ is bounded, }\ \ \forall  i=1,\dots,k.$$
In particular $|x_{i,p}|\rightarrow 0$, $\forall i=1,\dots,k$, as $p\rightarrow +\infty$, so that $\mathcal{S}=\{O\}$.
\end{proposition}
\begin{proof}
W.l.g. we can assume that for each $i=1,\ldots,k$ either $(x_{i,p})_p\subset \mathcal{N}_p^{+}$ or $(x_{i,p})_p\subset \mathcal{N}_p^{-}$. We prove the result in the case $(x_{i,p})_p\subset \mathcal{N}_p^{+}$, the other case being similar. Moreover in order to simplify the notation we drop the dependence on $i$ namely we set
$x_{p}:=x_{i,p}$ and $\mu_{p}:=\mu_{i,p}$.
\\
Let $h:=|G|$ and let us denote by $g^j$, $j=0,\dots, h-1$, the elements of $G$. We consider the rescaled nodal domains
\[
\widetilde{\mathcal{N}_{p}}^{j,+} :=\{x\in\mathbb R^2\ : \ \mu_p x +g^jx_p\in \mathcal N_p^+\}
%\frac{\mathcal{N}^+_{p}-g^j x_{p}^+}{\mu_{p}^+}
,\ \ j=0,\dots, h-1,\]
 and the rescaled functions $z_{p}^{j,+}(x): \widetilde{\mathcal{N}_{p}}^{j,+}\rightarrow\R$ defined by
\begin{equation}\label{z_j} z_{p}^{j,+}(x):=\frac{p}{u_{p}^+(x_{p})}\left( u_{p}^+(\mu_{p} x+g^jx_{p})-u_{p}^+(x_{p}) \right), \ \ j=0,\dots, h-1.\end{equation}
 Hence it's not difficult to see (as in Corollary \ref{cor:nonvedoNL}) that each $z_{p}^{j,+}$ converges to $U$ in $C^2_{loc}(\mathbb R^2)$ as $p\rightarrow \infty$, where $U$ is the function in \eqref{v0}.\\
Assume by contradiction that there exists a sequence $p_n\rightarrow +\infty$ as $n\rightarrow +\infty$ such that $\frac{|x_{p_n}|}{\mu_{p_n}}\rightarrow + \infty$.  Then, since the $h$ distinct points $g^j x_{p_n}$, $j=0,\ldots, h-1$, are the vertex of a regular polygon centered in $O$,   $d_n:=|g^j x_{p_n}-g^{j+1}x_{p_n}|=2\widetilde d_n \sin{\frac{\pi}{h}}$, where $\widetilde d_n:=|g^jx_{p_n}|$, $j=0,..,h-1$,  and so we also have that $\frac{d_n}{\mu_{p_n}}\rightarrow +\infty$ as $n\rightarrow +\infty$.
Let \begin{equation}\label{R_n}R_{n}:=\min\left\{\frac{d_n}{3},\frac{d(x_{p_n},\partial\Omega)}{2},\frac{d(x_{p_n},NL_{p_n})}{2}\right\},
\end{equation}
then  by construction  $B_{R_n}(g^j x_{p_n})\subseteq \mathcal{N}_{p_n}^+$ for $j=0,\dots,h-1$,
\begin{equation}\label{palleDisgiunte} B_{R_n}(g^j x_{p_n})\cap B_{R_n}(g^l x_{p_n}) =\emptyset,\ \ \mbox{ for }j\neq l
\end{equation}
and
\begin{equation}\label{invadeR2}
%\mbox{" }\  B_{\frac{R_n}{\mu_{p_n}}}(0)\rightarrow \R^2 \ \mbox{ %''}.
\frac{R_n}{\mu_{p_n}}\rightarrow  +\infty\quad \mbox{ as $n\rightarrow +\infty$.}
\end{equation}
Using \eqref{invadeR2}, the convergence of $z_{p_n}^{j,+}$ to $U$,  \eqref{RemarkMaxCirca1} and Fatou's lemma, we have
\begin{eqnarray}\label{betterEstimate}
8\pi &=&\int_{\R^2}e^{U}\,dx \leq \lim_n \int_{B_{\frac{R_n}{\mu_{p_n}}}(0)}\left|1+\frac{z_{p_n}^{j,+}}{p_n} \right|^{(p_n+1)}dx =\lim_n \frac{p_n}{\left|u_{p_n}^+(x_{p_n})\right|^2}  \int_{B_{R_n}(g^jx_{p_n})} \left| u_{p_n}^+\right|^{(p_n+1)}dx\nonumber\\
&\stackrel{\eqref{RemarkMaxCirca1}}{\leq}&
 \lim_n p_n\int_{B_{R_n}(g^jx_{p_n})} \left| u_{p_n}^+\right|^{(p_n+1)}dx.
\end{eqnarray}
Summing on $j=0,\dots, h-1$, using \eqref{palleDisgiunte}, \eqref{assumptionEnergyINIZIO}, \eqref{energylimitLowerPosNeg} and \eqref{energiaSuSoluzioni} we get:
\begin{eqnarray*}
h\cdot 8\pi
&\leq &
\lim_n\ p_n \sum_{j=0}^{h-1} \int_{B_{R_n}(g^jx_{p_n})} \left| u_{p_n}^+\right|^{(p_n+1)}dx
\stackrel{\mbox{\eqref{palleDisgiunte}}}{\leq}
\lim_n\   p_n\int_{\mathcal{N}_{p_n}^+} \left| u_{p_n}^+\right|^{(p_n+1)}dx
\\
&= &
 \lim_n\left(   p_n\int_{\Omega} \left| u_{p_n}\right|^{(p_n+1)}dx-\   p_n\int_{\mathcal{N}_{p_n}^-} \left| u_{p_n}^-\right|^{(p_n+1)}dx\right)
\stackrel{\eqref{assumptionEnergyINIZIO} + \eqref{energylimitLowerPosNeg}}\leq 
  \left(\alpha-1\right)\ \cdot 8\pi e
\stackrel{\alpha<5}< 
\  4 \ \cdot 8\pi e,
  \end{eqnarray*}
  which  contradicts the assumption $|G|\geq 4 e$.
  \end{proof}

\begin{remark}\label{rem:buonNormaInfinitoAbbassaSimmetria} 
If we knew that $\|u_p\|_{\infty}\geq \sqrt{e}$, then we would obtain a better estimate in \eqref{betterEstimate}, and so Proposition \ref{MaxVaazero} would hold under the weaker symmetry assumption $|G|\geq 4$  (recall that $|G|\geq 4$ is the assumption under which one can prove Theorem \ref{thm:existenceINTRO}).
\end{remark}

It is also possible to prove the following (see \cite[Corollary 3.5]{DeMarchisIanniPacellaJEMS} for more details):
\begin{corollary} \label{regioneNodaleInterna} Under the assumptions of Theorem \ref{TeoremaPrincipaleCasoSimmetrico}
\begin{itemize}
\item[$(i)$] $O\in \mathcal{N}_p^+$ for $p$ large.
\item[$(ii)$] Let $i\in\{1,\dots,k\}$ then $x_{i,p}\in\mathcal{N}_p^+$ for  $p$ large.
\end{itemize}
\end{corollary}
%
%
%\begin{proof}
%By the properties of the solutions $(u_p)$ we know that the nodal line  $NL_p$ is the boundary of a domain containing $O$ in his interior. Hence  if $O\not\in\mathcal{N}_{p_n}^+$ for a sequence $p_n\rightarrow +\infty$ as $n\rightarrow +\infty$, it would follow that
%\begin{equation}\label{distanza}
%d(x_{p_n}^+,NL_{p_n})\leq |x_{p_n}^+|.
%\end{equation}
%Dividing by $\mu_{p_n}^+$ and passing to the limit, from Corollary \ref{cor:nonvedoNL} (remember that $x_{p_n}^+$ has the role of $x_{1,p_n}$ in the general Theorem \ref{thm:x1N}) we get that
%\[\frac{|x_{p_n}^+|}{\mu_{p_n}^+}\rightarrow +\infty\quad \mbox{ as }n\rightarrow +\infty,\]
%which is a contradiction  with
%  Proposition \ref{MaxVaazero}. So $(i)$ holds.
%
%To prove $(ii)$ let us argue again by contradiction assuming that for a
%sequence  $p_n\rightarrow +\infty$ as $n\rightarrow +\infty$, $u_{p_n}(x_{i,p_n})<0$ holds for some $i\in\{1,\dots, k\}$. Then
%\[
%d(x_{i,p_n},NL_{p_n})\leq |x_{i,p_n}|
%\]
%so that, exactly with the same proof as in i), we reach a contradiction  with Proposition \ref{MaxVaazero}. So $(ii)$ holds.
%\end{proof}
%
%
\begin{proposition}\label{prop:bark=1} Under the assumptions of Theorem \ref{TeoremaPrincipaleCasoSimmetrico},
the maximal number $k$  of families of points $(x_{i,p})$, $i=1,\ldots, k$, for which $(P^k_1)$, $(P^k_2)$ and  $(P^k_3)$ hold is $1$.
\end{proposition}
\begin{proof}
Let us assume by contradiction that $k > 1$ and set $x^+_p=x_{1,p}$. For a family $(x_{j,p})$, $j\in\{2,\ldots, k\}$ by  Proposition \ref{MaxVaazero}, there exists $C>0$ such that
\[
\frac{|x_{1,p}|}{\mu_{1,p}}\leq C\quad\textrm{and}\quad\frac{|x_{j,p}|}{\mu_{j,p}}\leq C.
\]
Thus, since by definition $\mu^+_p=\mu_{1,p}\leq \mu_{j,p}$, also
\[
\frac{|x_{1,p}|}{\mu_{j,p}}\leq C.
\]
Hence
\[
\frac{|x_{1,p}-x_{j,p}|}{\mu_{j,p}}\leq\frac{|x_{1,p}|+|x_{j,p}|}{\mu_{j,p}}\leq C,
\]
which  contradicts $(\mathcal{P}_1^k)$ when $p\rightarrow +\infty$.
\end{proof}

Then we easily get
\begin{proposition}\label{prop:MinVaazero}
 Under the assumptions of Theorem \ref{TeoremaPrincipaleCasoSimmetrico} there exists $C>0$ such that
\begin{equation}\label{boundDaQ1}
\frac{|x_p|}{\mu (x_p)}\leq C
\end{equation}
for any family $(x_p)_p\subset \Omega$,  where $\mu (x_p)$ is defined as in \eqref{defmup}.
In particular, since by  \eqref{mumenozero} $\mu_p^-\rightarrow 0$, then $|x_p^-|\rightarrow 0$.
\end{proposition}
\begin{proof}
\eqref{boundDaQ1}  holds for $x_{p}^+$ by Proposition \ref{MaxVaazero}. Moreover $k=1$ by Proposition \ref{prop:bark=1}, so applying $(\mathcal{P}_3^1)$ to the points $(x_p)$, for $x_p\neq x_p^+$, we have
\[
\frac{|x_p-x_{p}^+|}{\mu (x_p)}\leq C.
\]
By definition, $\mu_p^+\leq \mu (x_p)$, hence we get
\[
\frac{|x_p|}{\mu (x_p)}\leq \frac{|x_p-x_{p}^+|}{\mu (x_p)}+\frac{|x_{p}^+|}{\mu (x_p)}\leq\frac{|x_p-x_{p}^+|}{\mu (x_p)}+\frac{|x_{p}^+|}{\mu_p^+}\leq C.
\]
\end{proof}

\begin{lemma}\label{lemma:MaxVaazeroVelocemente}  Let the  assumptions of Theorem \ref{TeoremaPrincipaleCasoSimmetrico} be satisfied and  let $(x_p)\subset\Omega$ be such that $p|\upp(x_p)|^{p-1}\to +\infty$ and $\mu (x_p)$ be as in \eqref{defmup}. Assume also that the rescaled functions $v_p(x):=\fr{p}{\upp(x_{p})}(\upp(x_{p}+\mu (x_{p}) x)-\upp(x_{p}))$ converge to $U$ in $C^2_{loc}(\R^2\setminus\{-\lim_{p}\tfrac{x_p}{\mu (x_p)}\})$ as $p\rightarrow +\infty$ ($U$ as in \eqref{v0}). Then
\begin{equation}\label{limiteZero}
\frac{|x_p|}{\mu (x_p)}\to 0\ \ \mbox{ as }\ p\rightarrow +\infty.
\end{equation}
As a byproduct we deduce that $v_p\rightarrow U$ in $C^2_{loc}(\mathbb{R}^2\setminus\{0\})$, as $p\rightarrow +\infty$.
\end{lemma}
\begin{proof}
%
% � un sottocaso del caso successivo, quindi lo tolgo
%
%First we show that $(x^+_{p})$ satisfies \eqref{limiteZero}. By  %Proposition \ref{MaxVaazero}  we know that $\frac{|x^+_{p}|}{\mu^+_{p}}\leq C$. Assume by contradiction that $\frac{|%x^+_{p}|}{\mu^+_{p}}\to\ell>0$ and consider the rescaling
%\bel\label{vip}
%v_{p}^+(x)=\fr{p}{\upp(x^+_{p})}(\upp(x^+_{p}+\mu^+_{p} x)-\upp(x^+_{p})).
%\eel
%Then, for any $g\in G$,
%\[
%v^+_{p}(\fr{gx^+_{p}-x^+_{p}}{\mu^+_{p}})=\fr{p}{\upp(x^+_{p})}(\upp(gx^+_{p})-\upp(x^+_{p}))=0
%\]
%by the symmetry of $u_p$. On the other side, since $\fr{|gx^+_{p}-x^+_{p}|}{\mu^+_{p}}=C_G\fr{|x^+_{p}|}{\mu^+_{p}}\to C_G\ell$ (where $C_G$ is a positive constant depending only on $G$), by the convergence of $v_p^+$ to $U$ we have
%\[
%v^+_{p}(\fr{gx^+_{p}-x^+_{p}}{\mu^+_{p}})\to U(x^+_{\infty})<0\qquad\textrm{as $p\to+\infty$},
%\]
%since $x^+_{\infty}:=\lim_p \fr{gx^+_{p}-x^+_{p}}{\mu^+_{p}}$  and $|x^+_{\infty}|=C_G\ell>0$.
%So we have reached a  contradiction.\\

%Next we consider a sequence $(x_p)_p\subset\Omega$ such that $\mu_p\to 0$ and that $v_p\to U$ in $C^1_{loc}(\R^2\setminus\{-\lim_{p}\tfrac{x_p}{\mu_p}\})$ as $p\rightarrow +\infty$.
%
By Proposition \ref{prop:MinVaazero} we know that $\frac{|x_p|}{\mu (x_p)}\leq C$. Assume by contradiction that $\fr{|x_p|}{\mu (x_p)}\to\ell>0$. Let $g\in G$ such that  $|x_{p}-g x_{p}|=C_g|x_{p}|$ with constant $C_g> 1$ (such a $g$ exists because $G$ is a group of rotation about the origin).
Hence
\[
\frac{|x_{p}-g x_{p}|}{\mu (x_{p})}=C_g\frac{|x_{p}|}{\mu (x_{p})}\rightarrow C_g\ell > \ell.
\]
Then $x_{0}:=\lim_{p\rightarrow +\infty} \fr{g x_p-x_p}{\mu (x_p)} \in\R^2\setminus\{-\lim_{p}\tfrac{x_p}{\mu (x_p)}\}$ and so by the $C^2_{loc}$ convergence we get
\[
v_{p}(\fr{g x_p-x_p}{\mu (x_p)})\to U(x_0)<0\quad\textrm{as $p\to+\infty$}.\]
On the other side, for any $g\in G$, one also has
\[
v_{p}(\fr{gx_{p}-x_{p}}{\mu (x_{p})})=\fr{p}{\upp(x_{p})}(\upp(gx_{p})-\upp(x_{p}))=0,
\]
by the symmetry of $u_p$ and this gives a contradiction.
\end{proof}

\

\

\subsection{Asymptotic analysis about the minimum points $x_p^-$ and study of $NL_p$}
\label{Subection:GSymmetricMin}

\

Proposition \ref{prop:bark=1} implies that $(\mathcal{P}_3^1)$ holds, from which 
\begin{equation}\label{minimVaAZerotilde}
\frac{|x_p^+-x_p^-|}{\mu_p^-}\leq C.
\end{equation}
with $\mu_p^-$ as in \eqref{mumenozero}.  Moreover, since we already know that $\frac{d(x_p^+,NL_p)}{\mu_p^+}\rightarrow +\infty$ as $p\rightarrow +\infty,$ we deduce that $\frac{|x_p^+-x_p^-|}{\mu_p^+}\rightarrow +\infty$ as $ p\rightarrow +\infty$, and in turn by \eqref{minimVaAZerotilde} we get
\begin{equation}
\label{rapportoMu}
\frac{\mu_p^+}{\mu_p^-}\rightarrow 0 \ \ \mbox{ as }\ p\rightarrow +\infty.
\end{equation}

Note that \eqref{minimVaAZerotilde} and \eqref{rapportoMu} more generally hold for any family of points $(x_p)$ such that $u_p(x_p)<0$ and $p|u_p(x_p)|^{p-1}\rightarrow +\infty$.\\
By Proposition \ref{prop:MinVaazero} we  have
\bel\label{minimVaAZero}
 \frac{|x_p^-|}{\mu_p^-}\leq C,
\eel
so there are two possibilities: either $\fr{|x_p^-|}{\mu_p^-}\to\ell>0$ or $\fr{|x_p^-|}{\mu_p^-}\to0$ as $p\rightarrow +\infty$, up to subsequences. We will exclude the latter case. We start with a preliminary result:
\begin{lemma} \label{Lemma:scalingPreliminareOrigine}
For $x\in \frac{\Omega}{|x_p^-|}:=\{y\in\mathbb R^2\ :\ y |x_p^-|\in\Omega \}$ let us define the rescaled function
\[
w_p^-(x):=\frac{p}{u_p(x_p^-)}\left(u_p(|x_p^-|x) -u_p(x_p^-)\right).
\]
Then
\begin{equation}
\label{v1pv0versione2}
w_p^-\rightarrow \gamma\ \mbox{  in }\ C^2_{loc}(\R^2\setminus\{0\})\ \mbox{ as }\ p\rightarrow +\infty,
\end{equation}
 where $\gamma\in C^2(\R^2\setminus\{0\})$, $\gamma\leq0$, $\gamma(x_{\infty})=0$ for a point $x_{\infty}\in\partial B_1(0)$ and it is a solution to
 \[
 -\Delta \gamma =\ell^2 e^{\gamma} \ \mbox{ in }\ R^2\setminus\{0\}.
 \]
In particular $\gamma\equiv 0$ when  $\ell =0$.
\end{lemma}
\begin{proof}
\eqref{minimVaAZero} implies that $|x_p^-|\rightarrow 0$ as $p\rightarrow +\infty$, so it follows that the set $\frac{\Omega}{|x_p^-|}\to\R^2$ as $p\to+\infty$.

By definition we have
\begin{equation}\label{conB}
w_p^-\leq 0,\  \ \ \  w_p(\frac{x_p^-}{|x_p^-|})=0 \ \ \ \mbox{ and } \ \ \ w_{p}^-=-p\quad\textrm{on $\partial \left( \frac{\Omega}{|x_p^-|} \right).$}
\end{equation}
Moreover, for $x\in \frac{\Omega}{|x_p^-|}$ we define $\xi_p:= |x_p^-|x$ and $\mu_{\xi_p}$ as $\mu_{\xi_p}^{-2}:= p|u_p(\xi_p)|^{p-1}$. Thanks to \eqref{problem1} we then have
\begin{equation}
\label{eq:RiscLineaNod}
|-\Delta w_p^-(x) |= \frac{p|x_p^-|^2|\upp(\xi_p)|^p}{|\upp(x^-_p)|}= \frac{|\upp(\xi_p)|}{|\upp(x^-_p)|}\,\frac{|x^-_p|^2}{\mu_{\xi_p}^2}\leq c_\infty \frac{|x^-_p|^2}{\mu_{\xi_p}^2},
\end{equation}
where  $c_\infty:=\lim_p\|\upp\|_\infty$.
Then, observing that $\frac{|x_p^-|}{\mu_{\xi_p}}\leq \frac{C}{|x|}$  by  Proposition \ref{prop:MinVaazero} applied to $\xi_p$, we have
\be
|-\Delta w_p^-(x) |
\leq \frac{c_\infty C^2}{|x|^2}.
\ee

Namely for any $R>0$
\begin{equation}\label{boundlaplacianonew}|-\Delta w_p^-|\leq c_\infty C^2R^2\ \ \  \  \mbox{ in } \frac{\Omega}{|x_p^-|}\setminus B_{\frac{1}{R}}(0).\end{equation}

So, similarly as in the proof of Theorem \ref{teo:BoundEnergia}-\emph{$(ii)$} (using now that  $w_p^-(\frac{x_p^-}{|x_p^-|})=0$), it follows that  for any $R>1$ ($\frac{x_p^-}{|x_p^-|}\in \partial B_1(0)\subset B_R(0)\setminus B_{\frac{1}{R}}(0)$ for $R>1$), $w_p^-$ is uniformly bounded in $B_R(0)\setminus B_{\frac{1}{R}}(0)$.

After passing to a subsequence, standard elliptic theory applied to the following equation
\bel\label{w-p}
-\lap w^-_p(x)=\frac{|x^-_p|^2}{(\mu^-_p)^2}\left(1+\frac{w^-_p(x)}{p}\right)\left|1+\frac{w^-_p(x)}{p}\right|^{p-1}
\eel
gives that $w_{p}^-$ is bounded in $ C^{2,\alpha}_{loc}(\mathbb R^2\setminus\{0\})$ .
Hence \eqref{v1pv0versione2} and the properties of $\gamma$ follow.
\\

In particular when $\ell =0$ it follows that $\gamma$ is harmonic in $\R^2\setminus\{0\}$ and $\gamma(x_\infty)=0$ for some point $x_\infty\in\partial B_1(0)$, therefore by the maximum principle we obtain $\gamma\equiv0$.
\end{proof}

\

\

\begin{proposition}\label{prop:NLp+l>0}
There exists $\ell>0$ such that
\[
\fr{|x^-_p|}{\mu_p^-}\to\ell \ \ \ \mbox{ as }\ p\rightarrow +\infty.
\]
\end{proposition}

\begin{proof}
By  Proposition \ref{prop:MinVaazero} we know that $\fr{|x^-_p|}{\mu^-_p}\to\ell\in[0,+\infty)$ as $p\rightarrow +\infty$. Let us suppose by contradiction that $\ell=0$.  Then  Lemma \ref{Lemma:scalingPreliminareOrigine} implies that \begin{equation}\label{wVaaZeroEqua}
w_p^-\rightarrow 0\ \mbox{ in }\ C^2_{loc}(\R^2\setminus\{0\}) \ \ \ \mbox{ as }\ p\rightarrow +\infty.
\end{equation}

\

By \eqref{problem1}, applying the divergence theorem in $B_{|x^-_p|}(0)$ we get
\bel\label{ABCassurdo}
p\int_{\partial B_{|x^-_p|}(0)}\nabla \upp(y)\cdot\fr{y}{|y|}\,d\sigma(y)=p\int_{B_{|x^-_p|}(0)\cap\mathcal{N}^-_p}|\upp(x)|^p\,dx-p\int_{B_{|x^-_p|}(0)\cap\mathcal{N}^+_p}|\upp(x)|^p\,dx.
\eel
Scaling $\upp$ with respect to $|x^-_p|$ as in Lemma \ref{Lemma:scalingPreliminareOrigine}, by \eqref{wVaaZeroEqua} we obtain
\begin{eqnarray}\label{primo}
&&\left|p\int_{\partial B_{|x^-_p|}(0)}\nabla \upp(y)\cdot\fr{y}{|y|}\,d\sigma(y)\right|=\left|p\int_{\partial B_{1}(0)}|x^-_p|\nabla \upp(|x^-_p|x)\cdot\fr{x}{|x|}\,d\sigma(x)\right|\nonumber\\
&&=\left|\int_{\partial B_{1}(0)}\upp(x^-_p)\,\nabla w^-_p(x)\cdot\fr{x}{|x|}\,d\sigma(x)\right|\,\leq\, |\upp(x^-_p)|\,2\pi \sup_{|x|=1}|\nabla w^-_p(x)|\,=\,o_p(1).
\end{eqnarray}

Now we want to estimate the right hand side in \eqref{ABCassurdo}. We first observe that scaling around $|x^-_p|$ with respect to $\mu^-_p$ we get
\begin{eqnarray}\label{secondo}
&&p\int_{B_{|x^-_p|}(0)\cap\mathcal{N}^-_p}|\upp(x)|^p\,dx=p\int_{B_{1}(0)\cap\fr{\mathcal{N}^-_p}{|x^-_p|}}|\upp(|x^-_p|y)|^p|x^-_p|^2\,dy\nonumber\\
&&\leq c_\infty \int_{B_1(0)\cap\fr{\mathcal{N}^-_p}{|x^-_p|}}\fr{|\upp(|x^-_p|y)|^{p-1}}{|\upp(x^-_p)|^{p-1}}\fr{|x^-_p|^2}{(\mu_p^-)^2}dy\,=\,o_p(1),
\end{eqnarray}
where in the last equality we have used that $\fr{|\upp(|x^-_p|y)|^{p-1}}{|\upp(x^-_p)|^{p-1}}\leq1$, since $|x^-_p|y\in \mathcal{N}^-_p$ and that by assumption $\fr{|x^-_p|}{\mu_p^-}\to0$ as $p\rightarrow +\infty$.

Next we claim that there exists $\bar p>1$ such that for any $p\geq \bar p$
\bel\label{inclusione}
B_{\mu^+_p}(x^+_p)\subset B_{|x^-_p|}(0).
\eel
Indeed, Corollary \ref{cor:nonvedoNL} implies that
\[+\infty = \lim_p \frac{d(x_p^+,NL_p)}{\mu_p^+}\leq \lim_p \frac{|x_p^+-x_p^-|}{\mu_p^+}\leq \lim_p\frac{|x_p^+|}{\mu_p^+}+\lim_p \frac{|x_p^-|}{\mu_p^+}=\lim_p\frac{|x_p^-|}{\mu_p^+},\]
where the last equality follows from Lemma \ref{lemma:MaxVaazeroVelocemente} (i.e. $\frac{|x_p^+|}{\mu_p^+}\rightarrow 0$). Hence for any $x\in B_1(0)$ we have
\[
\fr{|x^+_p+\mu^+_p x|}{|x^-_p|}\leq\fr{|x^+_p|}{|x^-_p|}+\fr{\mu^+_p}{|x^-_p|}\leq\fr{2\mu^+_p}{|x^-_p|}\to0\ \ \mbox{ as }\ p\to+\infty,
\]
and so \eqref{inclusione} is proved.\\

Hence by \eqref{inclusione} and scaling around $x^+_p$ with respect to $\mu^+_p$ we obtain
\bel\label{terzo}
p\int_{B_{|x^-_p|}(0)\cap\mathcal{N}^+_p}|\upp(x)|^p\,dx\geq p\int_{B_{\mu^+_p}(x^+_p)}|\upp(x)|^{p}dx=c_\infty\int_{B_1(0)}e^{U}dx+o_p(1).
\eel
Collecting \eqref{ABCassurdo}, \eqref{primo}, \eqref{secondo} and \eqref{terzo} we get clearly a contradiction.

\end{proof}

\

\

Next we show that the nodal line shrinks to the origin faster than $\mu_p^-$ as $p\rightarrow +\infty$.

\begin{proposition}
\label{prop:NodalLineShrinks}
We have
\[
\frac{\max\limits_{y_p\in NL_p}|y_p|}{\mu_p^-}\rightarrow 0 \ \ \mbox{ as }\ p\rightarrow +\infty.
\]
\end{proposition}
\begin{proof} By Proposition \ref{prop:NLp+l>0} it is enough to prove that
\[
\frac{\max\limits_{y_p\in NL_p}|y_p|}{|x_p^-|}\rightarrow 0 \ \ \mbox{ as }\ p\rightarrow +\infty.
\]
First we show that, for any $y_p\in NL_p$, the following alternative holds:
\begin{equation}\label{AlternativaEq}
\mbox{either }\ \ \frac{|y_p|}{|x_p^-|}\rightarrow 0\ \ \mbox{ or }\ \ \ \frac{|y_p|}{|x_p^-|}\rightarrow +\infty\ \ \mbox{ as }p\rightarrow +\infty.
\end{equation}

\

Indeed assume by contradiction that $\frac{|y_p|}{|x_p^-|}\rightarrow m\in (0,+\infty)$ as $p\rightarrow +\infty$.  Then $w_p^-(\frac{y_p}{|x_p^-|} )=-p\rightarrow -\infty$ as $p\rightarrow +\infty$. But we have proved in Lemma \ref{Lemma:scalingPreliminareOrigine}  that $w_p^-(\frac{y_p}{|x_p^-|} )\rightarrow\gamma(y_{m})\in\R$, where $y_{m}$ is such that $|y_{m}|=m>0$, and this gives a contradiction.

To conclude the proof we have then to exclude the second alternative in \eqref{AlternativaEq}.
For $y_p\in NL_p$, let us assume by contradiction that $\frac{|y_p|}{|x_p^-|}\rightarrow +\infty$ as $p\rightarrow +\infty$ and let us observe that
 \begin{equation}
 \label{unPuntoSullaLineaVaAZero}
 \exists\  z_p\in NL_p
 \ \mbox{  such that }\ \
\frac{|z_p|}{|x_p^-|}\rightarrow 0\ \mbox{ as }\ p\rightarrow +\infty.
\end{equation}
Indeed in the previous section we have shown that $O\in \mathcal N_p^+$,  hence there exists $t_p\in (0,1)$ such that $z_p:=t_p x_p^- \in NL_p$. Since $\frac{|z_p|}{|x_p^-|}<1$, by \eqref{AlternativaEq} we get \eqref{unPuntoSullaLineaVaAZero}.\\

Then for any $M>0$ there exists $\alpha_p^M\in NL_p$ such that $\frac{|\alpha_p^M|}{|x_p^-|}\rightarrow M$ as $p\rightarrow +\infty$ and this is in contradiction with \eqref{AlternativaEq}.
\end{proof}

\

\

Finally we can analyze the local behavior of $u_p$ around the minimum point $x_p^-$. Note that by Lemma \ref{lemma:MaxVaazeroVelocemente} and Proposition \ref{prop:NLp+l>0} we can already claim that the rescaling $v_p^-$ about $x_p^-$ (see \eqref{riscalataAboutMin}) cannot converge to  $U$ in $\R^2\setminus\{0\}$, where $U$ is the function in \eqref{v0}, indeed we have the following
\begin{proposition} Passing to a subsequence 
\label{prop:scalingNegativo}
\begin{equation}
v_p^-(x)\longrightarrow V(x-x_{\infty})\ \ \mbox{ in }C^2_{loc}(\mathbb R^2\setminus \{x_{\infty}\}),\mbox{ as }p\rightarrow +\infty
\end{equation}
where $V$ is the radial singular function in \eqref{espressioneEsplicitaV} which satisfies the  Liouville equation \eqref{LiouvilleSingularEquation} and $x_{\infty}\neq 0$ is like in 
%
%converges  to the function
%\[V_{\ell}(x):=\log\left(\frac{2\alpha^2\beta^{\alpha}|x-x_{\infty}|^{\alpha -2}}{(\beta^{\alpha}+|x-x_{\infty}|^{\alpha})^2} \right),
%\]
%where $\alpha=\alpha(\ell)=\sqrt{2\ell^2+4}$, $\beta=\beta(\ell)=\ell \left(\frac{\alpha+2}{\alpha-2} \right)^{1/\alpha}$, $x_{\infty}\in\mathbb R^2$, $|x_{\infty}|=\ell$ and $\ell=\lim_p\fr{|x^-_p|}{\mu^-_p}>0$.
%
%The function $V(x):=V_{\ell}(x+x_{\infty})$ is a radial singular solution of \eqref{LiouvilleSingularEquation} for $H=H(\ell)<0$.
\end{proposition}
\begin{proof}
Let us consider the translations of $v_p^-$:
\[
s_p^-(x):=v_p^-\left(x-\frac{x_p^-}{\mu_p^-} \right)=\fr{p}{\upp(x_{p}^-)}(\upp(\mu_p^- x)-\upp(x_p^-)),\quad \quad x\in \frac{\Omega}{\mu_p^-}
\]
which solve
\[
-\Delta s_p^-(x)=\left|1+\frac{s_p^-(x)}{p}    \right|^{p-1}\left(1+\frac{s_p^-(x)}{p}    \right),
%\frac{|u_p(\mu_p^- x)|^{p-1}u_p(\mu_p^- x)}{|u_p(x_p^-)|%^{p-1}u_p(x_p^-)},
\qquad\quad s^-_p(\fr{x^-_p}{\mu^-_p})=0,\qquad\quad s^-_p\leq 0.\]
Observe that $\frac{\Omega}{\mu_p^-}\rightarrow \mathbb R^2$ as $p\to+\infty$.\\ We claim that for any fixed $r>0$, $|-\Delta s_p^-|$ is bounded in $\frac{\Omega}{\mu_p^-}\setminus B_r(0)$. \\
Indeed Proposition \ref{prop:NodalLineShrinks} implies that if $x\in \frac{\mathcal{N}^+_p}{\mu_p^-}$,  then $|x|\leq\fr{\max\limits_{z_p\in NL_p}|z_p|}{\mu^-_p}<r,$ for $p$ large, hence
\[
\left( \frac{\Omega}{\mu_p^-}\setminus B_r(0) \right)
\subset\frac{\mathcal{N}^-_p}{\mu_p^-}\ \ \ \mbox{ for } p \mbox{ large}
\]
and so the claim follows observing that for $x\in \frac{\mathcal{N}^-_p}{\mu_p^-}$, then $|-\Delta s_p^-(x)|\leq1$.\\
Hence, by the arbitrariness of $r>0$, $s_p^-\rightarrow V$ in $C^2_{loc}(\mathbb R^2\setminus \{0\})$ as $p\rightarrow +\infty$ where $V$ is a solution of
$$
-\Delta V=e^{V}\ \  \mbox{ in }\ \R^2\setminus\{0\}
$$
which satisfies $V\leq 0$ and $V(x_{\ell})=0$ where $x_{\ell}:=\lim_p\fr{x^-_p}{\mu^-_p}$ and $|x_{\ell}|=\ell$ by Proposition \ref{prop:NLp+l>0}.
Moreover by virtue of Theorem \ref{teo:BoundEnergia}-\emph{(i)} and by \eqref{assumptionEnergyINIZIO} it can be seen that $e^{V}\in L^1(\R^2)$.\\
Observe that if $V$ was a classical solution of  $-\Delta V=e^V$ in the whole  $\R^2$ then necessarily $V(x)=U(x-x_{\ell})$. As a consequence $v_p^-(x)=s_p^-(x+\frac{x_p^-}{\mu_p^-})\rightarrow V(x+x_{\ell})=U(x)$ in $C^2_{loc}(\R^2\setminus\{-x_{\ell}\})$ as $p\rightarrow +\infty$. But then  Lemma \ref{lemma:MaxVaazeroVelocemente} would imply that $|x_{\ell}|=\frac{|x_p^-|}{\mu_p^-}\rightarrow 0$ as $p\rightarrow +\infty$, which is in contradiction with Proposition \ref{prop:NLp+l>0}. Thus, by \cite{ChenLi2,ChouWan1,ChouWan2} and the classification in \cite{ChenLi} we have that $V$ solves, for some $\eta >0$, the following entire equation
\[
\left\{
  \begin{array}{ll}
    -\lap V=e^{V}-4\pi\eta\delta_0 & \hbox{\textrm{in $\R^2$}} \\
    \int_{\R^2}e^{V}dx=8\pi(1+\eta), & \hbox{\:}
  \end{array}
\right.
\]
where $\delta_0$ denotes the Dirac measure centered at the origin.\\
By the classification given in \cite{PrajapatTarantello}, we have that either $V$ is radial, or $\eta\in\N$ and $V$ is $(\eta+1)$-symmetric. Actually it turns out that the energy bound \eqref{assumptionEnergyINIZIO} forces $V$ to be a radial solution, $V(r)$, satisfying
\[
\left\{
\begin{array}{lr}-V''-\frac{1}{r}V'=e^{V}\  \mbox{ in } (0, +\infty)\\
V\leq 0\\
V(\ell)=V'(\ell)=0
\end{array}
\right..
\]
The solution of this problem is
\[
V(r)=\log\left(\frac{2\alpha^2\beta^{\alpha}r^{\alpha -2}}{(\beta^{\alpha}+r^{\alpha})^2} \right),
\]
where $\alpha=\sqrt{2\ell^2+4}$ and $\beta=\ell \left(\frac{\alpha+2}{\alpha-2} \right)^{1/\alpha}$.
The conclusion follows observing that $v_p^-(x)=s_p^-\left( x+\frac{x_p^-}{\mu_p^-}\right)$ and setting $x_{\infty}=-x_{\ell}$.
\end{proof}

\

\subsection{Conclusion of the proof of Theorem \ref{TeoremaPrincipaleCasoSimmetrico}}

It follows combining all the previous results. More precisely \emph{$i)$} and \emph{$ii)$} follow from Proposition \ref{MaxVaazero} and Proposition \ref{prop:bark=1}. \emph{$iii)$}  is due to Proposition \ref{prop:MinVaazero}. \emph{$iv)$} is in Proposition \ref{prop:NodalLineShrinks}. Finally \emph{$v)$} comes from Proposition \ref{prop:scalingNegativo}.

\end{document}